\documentclass[final]{siamart1116}
\usepackage{amsmath}
\usepackage{amssymb}
\usepackage{accents}
\usepackage{mathtools}
\usepackage{graphicx}
\usepackage[export]{adjustbox}
\usepackage{array}
\usepackage{xcolor}
\usepackage{paralist}
\usepackage{caption}
\usepackage{subcaption}
\usepackage{tikz}
\usepackage{pgfplots}
\usetikzlibrary{plotmarks}
\usepackage{pgfplotstable}
\pgfplotsset{compat=1.12}
\usetikzlibrary{external}
\usepackage{wrapfig}
\usepackage{placeins}
\usepackage{siunitx}
\usepackage{marvosym}
\usepackage{todonotes}
\usepackage[english]{babel}
\usepackage{wasysym}
\usepackage{multirow}
\usepackage{tabu}
\usepackage{enumitem}
\usepackage{dsfont}
\usepackage{booktabs}
\usepackage{lipsum}
\usepackage{amsfonts}
\usepackage{graphicx}
\usepackage{epstopdf}
\usepackage{algorithmic}

\ifpdf
  \DeclareGraphicsExtensions{.eps,.pdf,.png,.jpg}
\else
  \DeclareGraphicsExtensions{.eps}
\fi

\numberwithin{theorem}{section}
\newtheorem{remark}[theorem]{Remark}

\usepackage{chngcntr}
\counterwithin{figure}{section}
\numberwithin{equation}{section}
\numberwithin{table}{section}

\usepackage[
  algo2e,
  linesnumbered,
  ruled
]{algorithm2e}
\SetKwProg{Fn}{Function}{\string:}{end}

\usepackage{hyperref}
\hypersetup{pdftex,colorlinks=true,allcolors=blue}
\usepackage{hypcap}

\makeatletter
\newcommand{\opnorm}{\@ifstar\@opnorms\@opnorm}
\newcommand{\@opnorms}[1]{%
  \left|\mkern-1.5mu\left|\mkern-1.5mu\left|
   #1
  \right|\mkern-1.5mu\right|\mkern-1.5mu\right|
}
\newcommand{\@opnorm}[2][]{%
  \mathopen{#1|\mkern-1.5mu#1|\mkern-1.5mu#1|}
  #2
  \mathclose{#1|\mkern-1.5mu#1|\mkern-1.5mu#1|}
}
\makeatother

\newcommand{\wavenumber}{\kappa}

\newcommand{\spanset}{\operatorname{span}}
\newcommand{\R}{\ \mathbb{R}}
\newcommand{\norm}[1]{{\left\|{#1}\right\|}}
 
\newcommand{\dist}{\operatorname{dist}} 
\newcommand{\divergence}{\operatorname{div}}

\newcommand{\range}{\operatorname{range}}
\newcommand{\spanlin}{\operatorname{span}}

\newcommand{\cest}{{c_\mathrm{est}(n_t, \varepsilon_{\mathrm{testfail}})}}

\newcommand{\ceff}{{c_\mathrm{eff}(n_t, \varepsilon_{\mathrm{testfail}})}}

\newcommand{\cestshort}{{c_\mathrm{est} } }
\newcommand{\ceffshort}{{c_\mathrm{eff} } }
\newcommand{\algotol}{\normalfont{\texttt{tol}}}
\newcommand{\analyticwidth}{W}

\definecolor{gelb}{RGB}{255,165,0}
\definecolor{oran}{RGB}{255,80,9}
\definecolor{lila}{RGB}{108,0,228}
\definecolor{gruen}{RGB}{0,184,0}
\definecolor{ei}{RGB}{251,255,238}
\definecolor{dunkelgruen}{RGB}{0,120,0}
\definecolor{warm_red}{RGB}{213,16,16}
\definecolor{grau}{RGB}{155,155,140}

\newcommand{\TheTitle}{Randomized Local Model Order Reduction} 
\newcommand{\TheAuthors}{Andreas Buhr and Kathrin Smetana}

\headers{\TheTitle}{\TheAuthors}

\title{{\TheTitle}\thanks{Submitted to the editors July 3, 2017.
\funding{Andreas Buhr was supported by CST Computer Simulation Technology AG.}}}

\author{
  Andreas Buhr\thanks{Institute for Computational and Applied Mathematics,
    University of M\"unster, Einsteinstra\ss e 62, 48149 M\"unster, Germany.
    (\email{andreas@andreasbuhr.de}).}
  \and
  Kathrin Smetana\thanks{Institute for Computational and Applied Mathematics,
    University of M\"unster, Einsteinstra\ss e 62, 48149 M\"unster, Germany; current address: Department of Applied Mathematics, University of Twente, P.O. Box 217, 7500 AE Enschede, The Netherlands.
    (\email{k.smetana@utwente.nl}).}
}

\usepackage{amsopn}

\ifpdf
\hypersetup{
  pdftitle={\TheTitle},
  pdfauthor={\TheAuthors}
}
\fi

\pdfoutput=1
  
\begin{document}
\maketitle
\begin{abstract}
In this paper we propose local approximation spaces for localized model order reduction procedures such as domain decomposition and multiscale methods. Those spaces are constructed from local solutions of the partial differential equation (PDE) with random boundary conditions, yield an approximation that converges provably at a nearly optimal rate, and can be generated at close to optimal computational complexity. In many localized model order reduction approaches like the generalized finite element method, static condensation procedures, and the multiscale finite element method local approximation spaces can be
constructed by approximating the range of a suitably defined transfer
operator that acts on the space of local solutions of the PDE. Optimal local approximation spaces that yield in general an exponentially convergent approximation are given by the
left singular vectors of this transfer operator [I.~Babu{\v{s}}ka and R.~Lipton 2011, K.~Smetana and A.~T.~Patera 2016]. However, the direct calculation of these singular vectors is computationally very expensive. In this paper, we propose an adaptive randomized algorithm based on methods from randomized linear algebra [N.~Halko et al. 2011], which constructs a local reduced space approximating the range of the transfer operator and thus the optimal local approximation spaces. 
Moreover, the adaptive algorithm relies on a probabilistic a posteriori error estimator for which we prove that it is both efficient and reliable with high probability. Several numerical experiments confirm the theoretical findings. 
\end{abstract}



\pagestyle{myheadings}
\thispagestyle{plain}
\markboth{A. BUHR AND K. SMETANA}{Randomized Local Model Order Reduction}

\section{Introduction}
Over the last decades (numerical) simulations based on partial differential equations (PDEs) have considerably gained importance in many (complex) applications. Model reduction is an indispensable tool for the simulation of complex problems where the use of standard methods such as finite elements (FE) and finite volumes is prohibitive. Examples for the latter are tasks where multiple simulation requests or real-time simulation response are desired, the (numerical) treatment of partial differential equations with rapidly varying and strongly heterogeneous coefficients, or simulations on very large or geometrically varying domains. Approaches developed to tackle such (complex) problems are localized model order reduction (localized MOR) approaches that are based on (combinations of) domain decomposition (DD) methods, multiscale methods, and the reduced basis method. This paper proposes local approximation spaces for interfaces or subdomains for local model order reduction procedures for linear, elliptic PDEs that yield a nearly optimally convergent approximation, are computationally inexpensive, and easy to implement.

Recently, local approximation spaces that are optimal in the sense of Kolmogorov \cite{Kol36} and thus minimize the approximation error among all spaces of the same dimension, have been introduced for subdomains $\Omega_{in}$ in \cite{BabLip11} and for interfaces $\Gamma_{in}$ in \cite{SmePat16}. 
To that end, an oversampling subdomain $\Omega$ which contains the target subdomain $\Omega_{in}$ or interface $\Gamma_{in}$ and whose boundary $\partial\Omega$ has a certain distance to the former is considered. Motivated by the fact that the global solution of the PDE satisfies the PDE locally, the space of harmonic functions --- that means all local solutions of the PDE with arbitrary Dirichlet boundary conditions --- is considered on the oversampling subdomain $\Omega$. Note that in general we expect an exponential decay of the higher frequencies of the Dirichlet boundary conditions to $\Omega_{in}$ or $\Gamma_{in}$. %
Therefore, we anticipate that already a local ansatz space of very small size should result in a very accurate approximation of all harmonic functions on $\Omega$. To detect the modes that still persist on $\Omega_{in}$ or $\Gamma_{in}$ a (compact) transfer operator is introduced that maps harmonic functions restricted to $\partial \Omega$ to harmonic functions restricted to $\Omega_{in}$ or $\Gamma_{in}$, respectively. The eigenfunctions of  the ``transfer eigenproblem'' --- the eigenvalue problem for the composition of the transfer operator and its adjoint --- span the optimal space which yields in general a superalgebraically and thus nearly exponentially convergent approximation. Recently, in \cite{TadPat17,Tad16} the results in \cite{BabLip11,SmePat16} have been generalized from linear differential operators whose associated bilinear form is coercive to elliptic, inf-sup stable ones.
However, computing say an FE approximation of these (continuous) optimal spaces by approximating the ``transfer eigenproblem'' requires first to solve the PDE on $\Omega$ for each FE basis function as Dirichlet boundary conditions on $\partial\Omega$ and subsequently to solve a dense eigenproblem of the size of the number of degrees of freedom (DOFs) on $\partial\Omega$. This is prohibitively expensive for many applications, especially for problems in three space dimensions. Applying the implicitly restarted Lanczos method as implemented in ARPACK \cite{Lehoucq1998} requires $\mathcal{O}(n)$ local solutions of the PDE in each iteration, where $n$ denotes the desired size of the local approximation space.

In this paper we propose to build local approximation spaces adaptively from local solutions of the PDE with (Gaussian) random boundary conditions. To give an intuition why randomly generated local approximation spaces may perform very well, we note that if we draw say $n$ independent random vectors which form the coefficients of FE basis functions on $\partial\Omega$ and apply the transfer operator, due to the extremely rapid decay of higher frequencies from $\partial\Omega$, the modes that still persist on $\Omega_{in}$ or $\Gamma_{in}$ will be very close to the optimal modes. 
In detail, based on methods from randomized linear algebra (randomized LA) \cite{halko2011finding,MaRoTy11} we propose an adaptive algorithm which iteratively enhances the reduced space by (local) solutions of the PDE for random boundary conditions and terminates when a probabilistic a posteriori error estimator lies below a given tolerance. We prove that after termination of the adaptive algorithm also the local approximation error is smaller or equal than the given tolerance with very high (given) probability. The respective probabilistic a posteriori estimator in this paper is an extension of a result in \cite{halko2011finding} and we show in addition, as one contribution of this paper, that the effectivity of the a posteriori error estimator can be bounded by a constant with high probability. By using the matrix representation of the transfer operator we exploit results from randomized LA \cite{halko2011finding} to prove that the reduced space produced by the adaptive algorithm yields an approximation that converges at a nearly optimal rate.\footnote{For a different analysis of the algorithm in \cite{halko2011finding,MaRoTy11} we refer to \cite{WitCan15}.} 
Thanks to this excellent approximation capacity the adaptive algorithm proposed in this paper thus only requires very few local solutions of the PDE in addition to the minimal amount required and is therefore computationally very efficient.
As one other (minor) contribution of this paper we extend the results for matrices in \cite{halko2011finding} to finite dimensional linear operators. We consider in this article parameter-independent PDEs. However, the extension to parameterized PDEs can be realized straightforward (see \cite{SmePat16,TadPat17}). Moreover, we assume here that the right-hand side of the PDE is given. If one wishes to construct local spaces for arbitrary right-hand sides prescribing random right-hand sides in the construction of local basis functions, as it is suggested in the context of numerical homogenization in \cite{Owh15,Owh17}, seems to be an attractive option.

Algorithms from randomized LA have got a steadily growing deal of attention in
recent years, especially for very large matrices for instance from problems in large-scale data analysis. Two of the most important benefits of randomization are that they can first result in faster algorithms, either in
worst-case asymptotic theory and/or numerical implementation, and that they allow very often for (novel) tight error bounds \cite{Mah11}. Finally, algorithms in randomized LA can often be designed to exploit modern computational
architectures better than classical numerical methods \cite{Mah11}. For open source software in randomized LA we refer for instance to \cite{VorMar15,Lietal17,EVBK16}. A very popular algorithm in randomized LA is the randomized singular value decomposition (SVD) (see for instance \cite{Sar06,MaRoTy11,RoSzTy09}), which yields a very accurate approximation of the (deterministic) SVD, getting however along with only $\mathcal{O}(n)$ applications of the matrix to random vectors. The randomized SVD can for instance rely on the matrix version of the adaptive algorithm we present in this paper (see \cite{halko2011finding,MaRoTy11}). Moreover, the latter shares a close relationship with methods in randomized LA that are based on the concept of dimension reduction, relying on a random linear map that performs an embedding into a low-dimensional space (see e.g. \cite{FrKaVe04,PRTV00,Sar06,MaRoTy11,RoSzTy09}). Other randomized algorithms employ element-wise sampling of the matrix --- for details we refer to the review in \cite{DriMah16} and the references therein --- or sampling of the columns or rows of the matrix \cite{FrKaVe98,FrKaVe04,DrKaMa06,RudVer07,BoMaDr09,DrMaMu08,DMMW12,DriMah16}. In both cases sampling is based on a certain probability distribution. In case of column sampling a connection to the low-rank approximations we are interested in in this paper can be set up via leverage scores \cite{Mah11,DMMW12}, where this probability distribution is based on (an approximation of) the space spanned by the best rank-$n$ approximation. In general this subcollection of columns or rows can then for instance be used to construct an interpolative decomposition or a CUR decomposition \cite{MaRoTy11,MahDri09,DrKaMa06b,CGMR05,DrMaMu08}. The matrix version of the adaptive algorithm we present in this paper can also be interpreted in the context of linear sketching: Applying the (input) matrix to a random matrix with certain properties results in a so-called ``sketch'' of the input matrix, which is either a smaller or sparser matrix but still represents the essential information of the original matrix (see for instance \cite{DriMah16,Woo14} and references therein); for instance, one can show that under certain conditions on the random matrix, the latter is an approximate isometry \cite{Ver12,Woo14}. The computations can then be performed on the sketch (see e.g. \cite{Woo14,Sar06}). Using structured random matrices such as a subsampled random Fourier or Hadamard transform or the fast Johnson-Lindenstrauss transform \cite{Sar06,WLRT08,Letal07,DMMW12,AilCha06,ClaWoo13} is particularly attractive from a computational viewpoint and yields an improved asymptotic complexity compared to standard methods. Finally, randomization can also be beneficial to obtain high-performant rank-revealing algorithms \cite{Mar15,DueGu17}. 

Using techniques from randomized LA has already been advocated in (localized) model order reduction approaches in other publications. In \cite{WanVou15} Vouvakis et.~al.~de- monstrated the potential of algorithms from randomized LA for domain decomposition methods by using adaptive, randomized techniques to approximate the range of discrete localized Dirichlet-to-Neumann maps in the context of a FETI-2$\lambda$ preconditioner. Regarding multi-scale methods the use of local ansatz spaces spanned by local solutions of the PDE with random boundary conditions is suggested in \cite{CEGL16} for the generalized multiscale finite element method (GMsFEM). Here, the reduced space is selected via an eigenvalue problem restricted to a space consisting of local solutions of the PDE with random boundary conditions. Based on results in \cite{MaRoTy06} an a priori error bound is shown, however, in contrast to our approach, it depends in general on the square root of the number of DOFs on the outer boundary $\partial\Omega$. Moreover, in contrast to \cite{CEGL16} we can formulate our procedure as an approximation of the optimal local approximation spaces suggested in \cite{BabLip11,SmePat16} and are thus able to provide a relation to the optimal rate. Eventually, the method proposed in \cite{CEGL16} either requires the dimension of the reduced space to be known in advance or the use of $\mathcal{O}(n)$ local solutions of the PDE in addition to the minimal amount required.
Finally, we note that in \cite{EftPat13} the local reduced space is constructed from local solutions of the PDE with a linear combination of discrete generalized Legendre polynomials with random coefficients as Dirichlet boundary conditions and in \cite{BEOR15} FE functions on $\partial\Omega$ with random coefficients are considered as boundary conditions. However, neither of the two articles takes advantage of the numerical analysis available for randomized LA techniques. 

The potential of applying algorithms from randomized LA in model order reduction has also already been demonstrated: In \cite{ZahNou16} a method for the construction of preconditioners of parameter-dependent matrices is proposed, which is an interpolation of the matrix inverse and is based on a projection of the identity matrix with respect to the Frobenius norm. Methods from randomized LA are used to compute a statistical estimator for the Frobenius norm. In \cite{HFPSetal14} a randomized SVD is employed to construct a reduced model for electromagnetic scattering problems. Finally, in \cite{AllKut16} the authors suggest to employ a randomized SVD to construct a reduced basis for the approximation of systems of ordinary partial differential equations.

There are many other choices of local approximation spaces in localized MOR approaches.
In DD methods reduced spaces on the interface or in the subdomains are for example chosen as the solutions of (local constrained) eigenvalue problems in component mode synthesis (CMS) \cite{Hur65,BamCra68,Bou92,HetLeh10} or (generalized) harmonic polynomials, plane waves, or local solutions of the PDE accounting for instance for highly heterogeneous coefficients in the Generalized Finite Element Method (GFEM) \cite{BaCaOs94,BaBaOs04,BabMel97,BabLip11}. 
In the Discontinuous Enrichment Method (DEM) \cite{FaHaFr01,FaKaTe10} local FE spaces are enriched by adding analytical or numerical free-space solutions of the homogeneous constant-coefficient counterpart of the considered PDE, while interelement continuity is weakly enforced via Lagrange multipliers. In multiscale methods such as
the multiscale FEM (MsFEM), the variational multiscale method (VMM), or the Local Orthogonal Decomposition Method (LOD) the effect of the fine scale on the coarse scale is either modeled analytically \cite{HFMQ98} or computed numerically by solving the fine-scale equations on local patches with homogeneous Dirichlet boundary conditions \cite{HouWu97,LarMal07,MalPet14}.

The reduced basis (RB) method has been introduced to tackle parameterized PDEs and prepares in a possibly expensive offline stage a low-dimensional reduced space which is specifically tailored to the considered problem in order to realize subsequently fast simulation responses for possibly many different parameters (for on overview see \cite{QuMaNe16,HeRoSt2016,Haa14}). Combinations of the RB method with DD methods have been considered in \cite{MadRon02,MadRon04,IaQuRo12,AnPaQu16,HuKnPa13,EftPat13,Smetana15,IaQuRo16,MaiHaa14,MarRozHaa14,BEOR15}. Here, intra-element RB approximations are for instance coupled by either polynomial Lagrange multipliers \cite{MadRon02,MadRon04}, generalized Legendre polynomials \cite{HuKnPa13}, FE basis functions \cite{IaQuRo16},  or empirical modes generated from local solutions of the PDE  \cite{EftPat13,MarRozHaa14,BEOR15} on the interface. In order to address parameterized multiscale problems the local approximation spaces are for instance spanned by eigenfunctions of an eigenvalue problem on the space of harmonic functions in \cite{EfGaHo13}, generated by solving the global parameterized PDE and restricting the solution to the respective subdomain in \cite{OhlSch15,AHOK12}, or enriched in the online stage by local solutions of the PDE, prescribing the insufficient RB solution as Dirichlet boundary conditions in \cite{OhlSch15,AHOK12}. Apart from that the RB method has also been used in the context of multiscale methods for example in \cite{Ngu08,HeZhZh15,AbdHen14}.

The remainder of this paper is organized as follows. In \cref{sect:optimal local spaces} we present the problem setting and recall the main results for the optimal local approximation spaces introduced in \cite{BabLip11,SmePat16}. The main contributions of this paper are developed in \cref{sect:randomized_la} where we propose an adaptive algorithm that generates local approximation spaces.
Moreover, we prove a priori and a posteriori error bounds and show that the latter is efficient. Finally, we present numerical results in \cref{sect:numerical_experiments} for the Helmholtz equation, stationary heat conduction with high contrast, and linear elasticity to validate the theoretical findings and draw some conclusions in \cref{sect:conclusions}.

\section{Optimal local approximation spaces for localized model order reduction procedures}\label{sect:optimal local spaces}

Let $\Omega_{gl} \subset \mathbb{R}^{d}$, $d=2,3$, be a large, bounded domain with Lipschitz boundary and assume that $\partial \Omega_{gl}= \Sigma_{D} \cup \Sigma_{N}$, where $\Sigma_{D}$ denotes the Dirichlet and $\Sigma_{N}$ the Neumann boundary, respectively. We consider a linear, elliptic PDE on $\Omega_{gl}$ with solution $u_{gl}$, where $u_{gl} = g_{D}$ on $\Sigma_{D}$ and satisfies homogeneous Neumann boundary conditions on $\Sigma_{N}$. Note that we consider here homogeneous Neumann boundary conditions to simplify the notation; non-homogeneous Neumann boundary conditions can be taken into account completely analogous to non-homogeneous Dirichlet boundary conditions. 
To compute an approximation of $u_{gl}$ we employ a domain decomposition or multiscale method combined with model order reduction techniques, which is why we suppose that $\Omega_{gl}$ is decomposed into either overlapping or non-overlapping subdomains. Then, depending on the employed  method, one may either require good  reduced spaces for the subdomains, the interfaces, or both. To fix the setting we thus consider the task to find a good reduced space either on a subdomain $\Omega_{in} \subsetneq \Omega \subset \Omega_{gl}$ with $\dist(\Gamma_{out},\partial \Omega_{in}) \geq \rho > 0$, $\Gamma_{out}:=\partial\Omega\setminus\partial \Omega_{gl}$ or an interface $\Gamma_{in} \subset \partial \Omega_{*}$, where $\Omega_{*} \subsetneq \Omega \subset \Omega_{gl}$ and $\dist(\Gamma_{out},\Gamma_{in})\geq \rho > 0$. Possible geometric configurations are illustrated in \cref{fig:illustration geometry}.

\begin{figure}[t]
\begin{center}
              \includegraphics[width=0.24\textwidth]{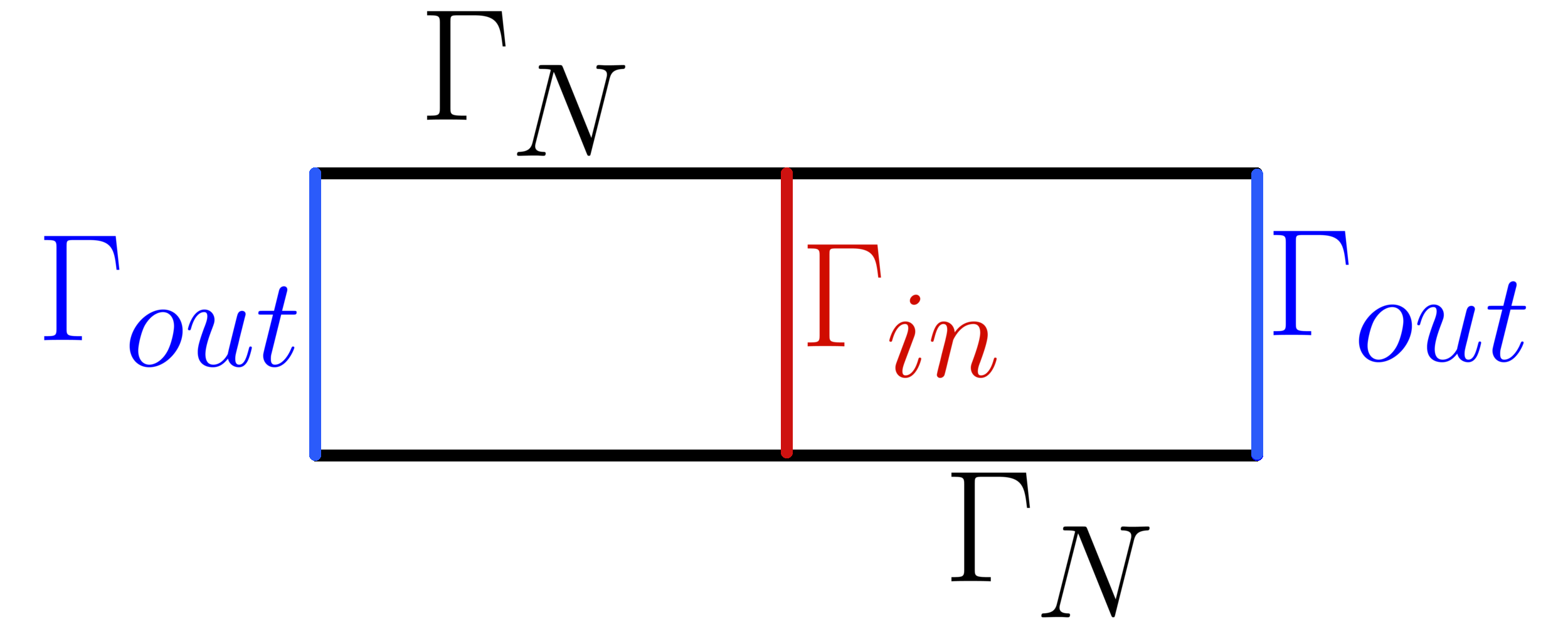} \quad
              \includegraphics[width=0.16\textwidth]{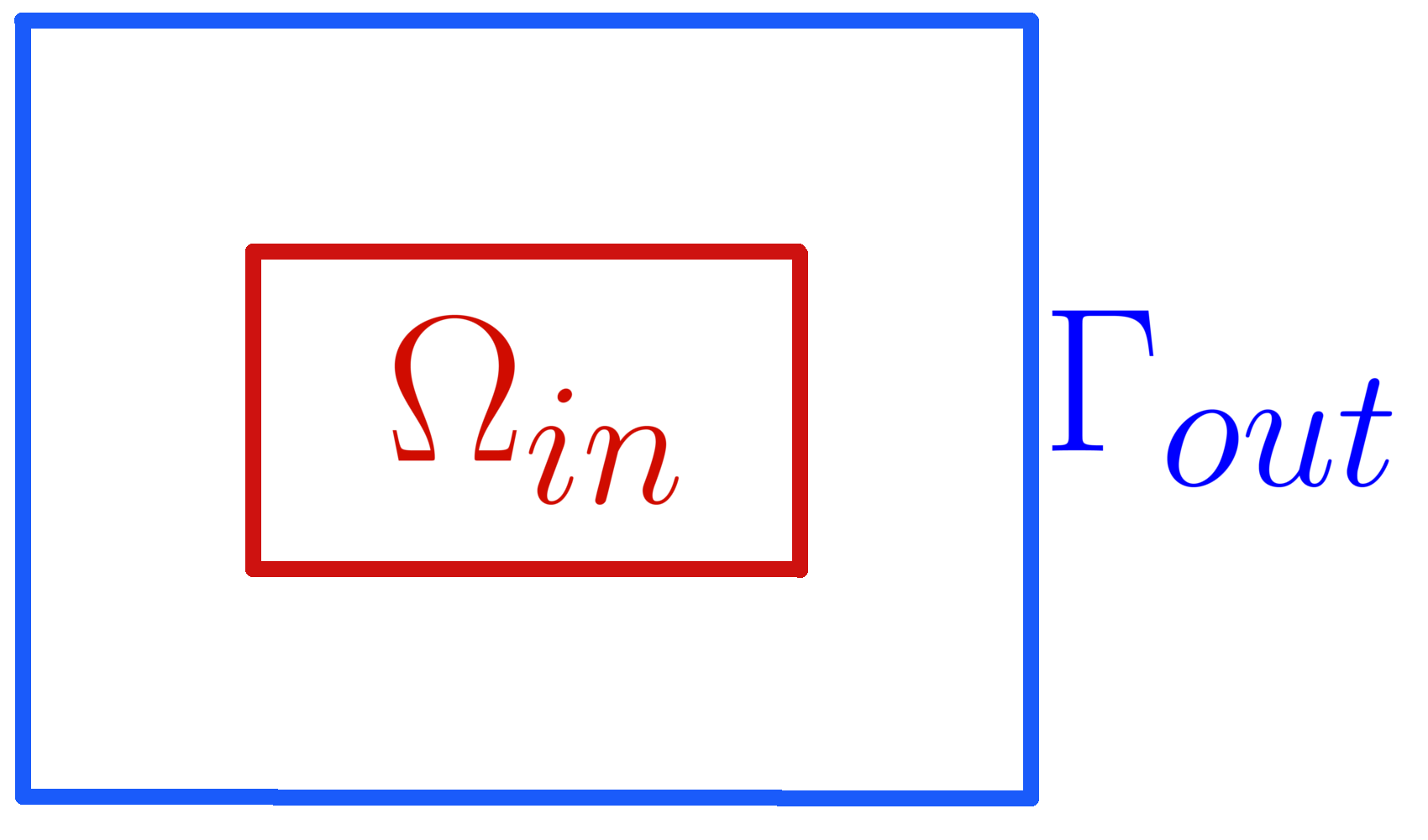}  
              \caption{\footnotesize Illustration of possible decompositions of $\Omega$ with respect to $\Gamma_{in}$ or $\Omega_{in}$.}\label{fig:illustration geometry}
              \end{center}
\vspace{-15pt}
\end{figure}

The challenge in constructing a good reduced space is the fact that although we know that $u_{gl}$ solves the PDE locally on $\Omega$ we do in general \emph{not} know the trace of $u_{gl}$ on $\partial\Omega$ a priori. Therefore, we consider the following problem on $\Omega$: For given $f \in \mathcal{X}_{0}'$ find $u \in \mathcal{X}:=\{ w \in [H^{1}(\Omega)]^{z}\, : \, w = g_{D} \enspace \text{on} \enspace \partial \Omega \cap \Sigma_{D}\}$ such that
\begin{equation}\label{eq:PDE}
\mathcal{A}u = f \quad \text{in} \enspace \mathcal{X}'_{0},
\end{equation}
for arbitrary Dirichlet boundary conditions on $\Gamma_{out}$, where $\mathcal{A}:[H^{1}(\Omega)]^{z} \rightarrow \mathcal{X}_{0}'$, $z=1,2,3$ is a linear, elliptic, and continuous differential operator and $\mathcal{X}_{0}'$ denotes the dual space of 
$
\mathcal{X}_{0}:= \{ v \in [H^{1}(\Omega)]^{z} \, : \, v|_{\Gamma_{out}} = 0, \enspace v|_{\Sigma_{D}\cap \partial \Omega} = 0 \}, \enspace z =1,2,3.
$
The latter is in turn equipped with the full $H^1$-norm.

By exploiting that the global solution $u_{gl}$ solves the PDE \cref{eq:PDE} locally, recently, optimal local approximation spaces have been introduced for subdomains in \cite{BabLip11} and for interfaces in \cite{SmePat16}.\footnote{The key concepts of the construction of optimal local approximation spaces can be nicely illustrated by means of separation of variables 
in a simple example as in \cite[Remark 3.3]{SmePat16}, see the supplementary materials section SM2.} As we aim at providing a good approximation for a whole set of functions, namely all functions that solve the PDE \cref{eq:PDE} locally, the concept of optimality of Kolmogorov \cite{Kol36} is used:

\begin{definition}[Optimal subspace in the sense of Kolmogorov]\label{definition:Kolmogorov}
Let $\mathcal{S},\mathcal{R}$ be Hilbert spaces, $\mathcal{T}: \mathcal{S} \rightarrow \mathcal{R}$ a linear, continuous operator, and $\mathcal{R}^{n}$ an $n$-dimensional subspace of $\mathcal{R}$. Then the Kolmogorov $n$-width of the image of the mapping $\mathcal{T}$ applied to the unit ball of $\mathcal{S}$ in $\mathcal{R}$ is given by
\begin{equation}\label{eq:n-width operator}
d_{n}(\mathcal{T}(\mathcal{S});\mathcal{R}) := \underset{\dim(\mathcal{R}^{n})=n}{\inf_{\mathcal{R}^{n}\subset \mathcal{R}}} \sup_{\psi \in \mathcal{S}} \inf_{\zeta \in \mathcal{R}^{n}} \frac{\|\mathcal{T}\psi - \zeta\|_{\mathcal{R}}}{\|\psi\|_{\mathcal{S}}} = \underset{\dim(\mathcal{R}^{n})=n}{\inf_{\mathcal{R}^{n}\subset \mathcal{R}}} \underset{\|\psi\|_{\mathcal{S}} \leq 1}{\sup_{\psi \in \mathcal{S}}} \inf_{\zeta \in \mathcal{R}^{n}} \|\mathcal{T}\psi - \zeta\|_{\mathcal{R}}.
\end{equation}
A subspace $\mathcal{R}^{n} \subset \mathcal{R}$ of dimension at most $n$ for which holds
\begin{equation*}
d_{n}(\mathcal{T}(\mathcal{S});\mathcal{R}) = \sup_{\psi \in \mathcal{S}} \inf_{\zeta \in \mathcal{R}^{n}} \frac{\|\mathcal{T}\psi - \zeta\|_{\mathcal{R}}}{\|\psi\|_{\mathcal{S}}}
\end{equation*}
is called an optimal subspace for $d_{n}(\mathcal{\mathcal{T}}(\mathcal{S});\mathcal{R})$.
\end{definition}
Being interested in all local solutions of the PDE motivates considering the space of $\mathcal{A}$-harmonic functions on $\Omega$
\begin{equation}\label{eq:space of harmonic functions}
\tilde{\mathcal{H}}:= \{ w \in [H^{1}(\Omega)]^{z}\, : \, \mathcal{A}w = 0 \enspace \text{in} \enspace \mathcal{X}_{0}', \enspace w = 0 \enspace \text{on} \enspace \Sigma_{D}\cap\partial\Omega\}, \enspace z=1,2,3.
\end{equation}
Note that first we restrict ourselves here to the case $f=0$, $g_{D}=0$, and $\partial\Omega_{in}\cap\Sigma_{D}=\emptyset$; the general case will be dealt with at the end of this subsection. 

As in \cite{BabLip11,SmePat16} we may then introduce a transfer operator $\mathcal{T}: \mathcal{S} \rightarrow \mathcal{R}$ for Hilbert spaces $\mathcal{S}$ and $\mathcal{R}$, where $\mathcal{S}= \{ w|_{\Gamma_{out}} \, : \, w \in \tilde{\mathcal{H}}\}$. In order to define appropriate range spaces $\mathcal{R}$ that ensure compactness of $\mathcal{T}$ and allow equipping $\mathcal{R}$ with an energy inner product, we first introduce for a domain $D \subset \Omega$ an orthogonal projection $P_{\ker(\mathcal{A}),D}: [H^{1}(D)]^{z} \rightarrow \ker(\mathcal{A})$ defined as $P_{\ker(\mathcal{A}),D}v := \sum_{k=1}^{\dim (\ker(\mathcal{A}))} (v,\eta_{k})_{quot}\,\eta_{k}$.
Here, $\eta_{k}$ is an orthonormal basis of $\ker(\mathcal{A})$ with respect to the $( \cdot , \cdot)_{quot}$ inner product, where the definition of the latter has to be inferred from the quotient space $\tilde{\mathcal{H}}|_D/\ker(\mathcal{A})$. To illustrate those definitions note that for instance for the Laplacian
$\ker(\mathcal{A})$ would be the constant functions and $( \cdot , \cdot)_{quot}$ would be the $L^{2}$-inner product on $D$. In the case of linear elasticity $\ker(\mathcal{A})$ would equal the six-dimensional space of the rigid body motions and $( \cdot , \cdot)_{quot}$ has to be chosen as the full $H^{1}$-inner product on $D$. We may then define the quotient space 
$
\mathcal{H} := \{ v - P_{\ker(\mathcal{A}),\Omega}(v) , \enspace v \in \tilde{\mathcal{H}}
\}
$
and specify the transfer operator. For $w \in \tilde{\mathcal{H}}$ we define $\mathcal{T}$ for interfaces or subdomains, respectively, as
\begin{equation}\label{eq:transfer_op_gamma}
\mathcal{T}(w|_{\Gamma_{out}})=\left(w-P_{\ker(\mathcal{A}),\Omega}(w)\right)|_{\Gamma_{in}}
\enspace \text{or} \enspace
\mathcal{T}(w|_{\Gamma_{out}})=\left(w - P_{\ker(\mathcal{A}),\Omega_{in}}(w)\right)|_{\Omega_{in}}
\end{equation}
and set
$\mathcal{R}= \{ w|_{\Gamma_{in}} \, : \, w \in \mathcal{H}\}$
or
$\mathcal{R}= \{ \left(w - P_{\ker(\mathcal{A}), \Omega_{in}}\right)|_{\Omega_{in}} \, : \, w \in \tilde{\mathcal{H}}\}$.

Some remarks are in order. In contrast to the definitions in \cite{BabLip11, SmePat16} we do not use a quotient space in the definition of the source space $\mathcal{S}$ as this would either significantly complicate the analysis of the randomized local spaces in \cref{sect:randomized_la} or require the construction of a suitable basis in $\mathcal{S}$ or its discrete counterpart, which can become computationally expensive. %
Thanks to the Caccioppoli inequality (see supplementary materials section SM1), which allows us to bound the energy norm of $\mathcal{A}$-harmonic functions on $\Omega_{in}$ or $\Omega_{*}$, respectively, by their $L^{2}$-norm on $\Omega$, it can then be proved that the operator $\mathcal{T}$ is compact (see \cite{BabLip11,BaHuLi14,SmePat16} for details), where $\Omega_{*}$ has been defined in the second paragraph of this section.\footnote{Note in this context that compactness of $\mathcal{T}$ as defined in \cref{eq:transfer_op_gamma} can be easily inferred from the compactness of the transfer operator acting on the quotient space $\tilde{\mathcal{H}}/\ker(\mathcal{A})$ as considered in \cite{BabLip11,BaHuLi14,SmePat16} by employing that the mapping $\mathcal{K}: \tilde{\mathcal{H}}|_{\Gamma_{out}} \rightarrow (\tilde{\mathcal{H}}/\ker(\mathcal{A}))|_{\Gamma_{out}}$ defined as $\mathcal{K}(v|_{\Gamma_{out}}):=(v - P_{\ker(\mathcal{A}),\Omega})|_{\Gamma_{out}}$ is continuous. }
Let finally $\mathcal{T}^{*}:\mathcal{R} \rightarrow \mathcal{S}$ denote the adjoint operator of $\mathcal{T}$. Then, the operator $\mathcal{T}^{*}\mathcal{T}$ is a compact, self-adjoint, non-negative operator that maps $\mathcal{S}$ into itself, and the Hilbert-Schmidt theorem and Theorem 2.2 in Chapter 4 of \cite{Pinkus85} yield the following result:
\begin{theorem}[Optimal local approximation spaces \cite{BabLip11,SmePat16}]\label{theorem:eigenvalue problem}
The optimal approximation space for $d_{n}(\mathcal{T}(\mathcal{S});\mathcal{R})$ is given by 
\begin{equation}\label{eq:optimal space}
\mathcal{R}^{n}:= \spanset \{\phi_{1}^{sp},...,\phi_{n}^{sp}\}, \qquad \text{where} \enspace \phi_{j}^{sp}=\mathcal{T}\varphi_{j},\quad j=1,...,n,
\end{equation}
and 
$\lambda_j$ are the largest $n$ eigenvalues and $\varphi_j$ the corresponding eigenfunctions
that satisfy the transfer eigenvalue problem: Find $(\varphi_{j},\lambda_{j}) \in (\mathcal{S},\mathbb{R}^{+})$ such that
\begin{align}\label{eq:transfer eigenvalue problem}
(\,\mathcal{T}\varphi_{j}\,,\,\mathcal{T}w\,)_{\mathcal{R}} = \lambda_{j} (\,\varphi_{j}\,,\,w\,)_{\mathcal{S}} \quad \forall w \in \mathcal{S}.
\end{align}
Moreover, the following holds: 
\begin{equation}\label{eq:n-width_equals_eigenvalue}
d_{n}(\mathcal{T}(\mathcal{S});\mathcal{R}) = \sup_{\xi \in \mathcal{S}} \inf_{\zeta \in \mathcal{R}^{n}} \frac{\| \mathcal{T}\xi - \zeta\|_{\mathcal{R}}}{\|\xi\|_{\mathcal{S}}} = \sqrt{\lambda_{n+1}}
\end{equation}
\end{theorem} 

If we have $\partial\Omega_{in}\cap\Sigma_{D}\neq\emptyset$ we do not subtract the orthogonal projection on $\ker(\mathcal{A})$ either in the definition of the transfer operator in \cref{eq:transfer_op_gamma} or the definition of the range space for subdomains. Next, for $f\neq 0$ but still $g_{D}=0$ we solve the problem: Find $u^{f} \in \mathcal{X}_{0}$ such that 
$
\mathcal{A}u^{f} = f \quad \text{in} \enspace \mathcal{X}_{0}',
$
and augment the space $\mathcal{R}^{n}$ either with $u^{f}|_{\Omega_{in}}$ or $u^{f}|_{\Gamma_{in}}$. To take non-homogeneous Dirichlet boundary conditions into account we consider the problem: Find $u^{g_{D}} \in \{ w \in [H^{1}(\Omega)]^{z}\, : \, w = g_{D} \enspace \text{on} \enspace \partial \Omega \cap \Sigma_{D}, \enspace w = 0 \enspace \text{on} \enspace\Gamma_{out}\}$, $z=1,2,3$, such that 
$
\mathcal{A}u^{g_{D}} = 0 \quad \text{in} \enspace \mathcal{X}_{0}'.
$
Finally, we may then define the optimal local approximation space for subdomains as 
\begin{equation}\label{eq:optimal space subdomain}
\mathcal{R}_{data,\ker}^{n}:=\spanlin \{ \phi_{1}^{sp},...,\phi_{n}^{sp}, u^{f}|_{\Omega_{in}}, u^{g_{D}}|_{\Omega_{in}},\eta_{1}|_{\Omega_{in}},\hdots,\eta_{\dim(\ker(\mathcal{A}))}|_{\Omega_{in}}\}
\end{equation}
and for interfaces as
\begin{equation}\label{eq:optimal space interface}
\mathcal{R}_{data,\ker}^{n}:=\spanlin \{ \phi_{1}^{sp},...,\phi_{n}^{sp}, u^{f}|_{\Gamma_{in}},u^{g_{D}}|_{\Gamma_{in}},\eta_{1}|_{\Gamma_{in}},\hdots,\eta_{\dim(\ker(\mathcal{A}))}|_{\Gamma_{in}}\},
\end{equation}
respectively, where $\{ \eta_{1},\hdots,\eta_{\dim(\ker(\mathcal{A}))}\}$ denotes a basis for $\ker(\mathcal{A})$. In case there holds $\partial\Omega_{in}\cap\Sigma_{D}\neq\emptyset$ we do not augment the space $\mathcal{R}^{n}$ with a basis of $\ker(\mathcal{A})$.

\subsection{Approximation of the transfer eigenvalue problem with Finite Elements; matrix form of the transfer operator}

In this subsection we show how an approximation of the continuous optimal local spaces $\mathcal{R}_{data,\ker}^{n}$ can be computed with the FE method and introduce the notation in this discrete setting required for the remainder of this paper. 

To that end, we introduce a partition of $\Omega$ such that $\Gamma_{in}$ or $\partial \Omega_{in}$ do not intersect any element of that partition. In addition, we introduce an associated conforming FE space $X \subset [H^{1}(\Omega)]^{z}$, $z=1,2,3$ with $\dim(X) = N$, a nodal basis $\{\psi_{1},...,\psi_{N}\}$ of $X$, the FE source space $S:=\{ v|_{\Gamma_{out}}\, : \, v \in X\}$ of dimension $N_{S}$, and the FE range space $R:=\{ (v - P_{ker(\mathcal{A}),\Omega}(v))|_{\Gamma_{in}}\, : \, v \in X\}$ or $R:=\{ (v - P_{ker(\mathcal{A}),\Omega_{in}})|_{\Omega_{in}}\, : \, v \in X\}$ with $\dim(R)=N_{R}$. Next, we define the space of discrete $A$-harmonic functions
\begin{equation}
\tilde{H} := \{ w \in X\, : \, Aw = 0 \enspace \text{in} \enspace X_{0}', \enspace w = 0 \enspace \text{on} \enspace \Sigma_{D}\cap\partial\Omega\},
\end{equation}
where $A:X \rightarrow X_{0}'$ is the discrete counterpart of $\mathcal{A}$ and $X_{0}'$ denotes the dual space of $X_{0}:= \{ v \in X \, : \, v|_{\Gamma_{out}} = 0, \enspace v|_{\Sigma_{D}\cap \partial \Omega} = 0 \}.$ We may then define the discrete transfer operator $T:S \rightarrow R$ for $w \in \tilde{H}$ as\footnote{Note that in the continuous setting the range space $\mathcal{R}$ is a subspace of the space $\widehat{\mathcal{R}}:=\{(v - P_{ker(\mathcal{A}),\Omega}(v))|_{\Gamma_{in}}, \enspace v \in [H^{1}(\Omega)]^{z}\}$, $z=1,2,3$ for interfaces and $\widehat{\mathcal{R}}:=\{(v - P_{ker(\mathcal{A}),\Omega_{in}}(v))|_{\Omega_{in}}, \enspace v \in [H^{1}(\Omega)]^{z}\}$, $z=1,2,3$ for subdomains. It can then be easily shown for the corresponding transfer operator $\widehat{\mathcal{T}}: \widehat{\mathcal{R}} \rightarrow\mathcal{S}$ which is defined identically as in \cref{eq:transfer_op_gamma} that there holds $d_{n}(\mathcal{T}(\mathcal{S});\mathcal{R}) = d_{n}(\widehat{\mathcal{T}}(\mathcal{S});\widehat{\mathcal{R}})$ and that the associated optimal approximation spaces are the same. This justifies the usage of the discrete range space as defined above.}  
\begin{equation}\label{eq: discrete transfer operator}
T(w|_{\Gamma_{out}})=\left(w-P_{\ker(\mathcal{A}),\Omega}(w)\right)|_{\Gamma_{in}} \enspace \text{or} \enspace
T(w|_{\Gamma_{out}})=\left(w - P_{\ker(\mathcal{A}),\Omega_{in}}(w)\right)|_{\Omega_{in}}.
\end{equation}

In order to define a matrix form of the transfer operator we introduce DOF mappings $\underline{B}_{S\rightarrow X} \in \mathbb{R}^{N \times N_{S}}$ and $\underline{B}_{X\rightarrow R} \in \mathbb{R}^{N_{R}\times N}$ that map the DOFs of $S$ to the DOFs of $X$ and the DOFs of $X$ to the DOFs of $R$, respectively. Moreover, we introduce the stiffness matrix $\underline{A}$ associated with the discrete operator $A$, where we assume that in the rows associated with the Dirichlet DOFs the non-diagonal entries are zero and the diagonal entries equal one. Note that in order to make the distinction between elements of the Hilbert spaces $S$ and $R$ and their coordinate representation in $\R^{N_S}$ and $\R^{N_R}$
explicit, we mark all coordinate vectors and matrices with an underline. By writing functions $\zeta \in S$ as $\zeta = \sum_{i=1}^{N_{S}}\underline{\zeta}_{i} \psi_{i}|_{\Gamma_{out}}$ 
and defining $\underline K_{\Omega_{in}}$ as the matrix of the orthogonal projection on $\ker(\mathcal{A})$ on $\Omega_{in}$,
we obtain the following matrix representation $\underline{T} \in \mathbb{R}^{N_{R} \times N_{S}}$ of the transfer operator for domains
\begin{eqnarray}\label{eq:matrix_form_transfer_operator}
\underline{T}\,\underline{\zeta} = 
\left(1 - \underline K_{\Omega_{in}} \right)
\underline{B}_{X\rightarrow R} \, \underline{A}^{-1} \underline{B}_{S\rightarrow X} \, \underline{\zeta}.
\end{eqnarray}
For interfaces, the projection on the quotient space is done before the index mapping.
There, with $\underline K_{\Omega}$ as the matrix of the orthogonal projection on $\ker(\mathcal{A})$ on $\Omega$,
the matrix representation of the transfer operator is given by
\begin{eqnarray}\label{eq:matrix_form_transfer_operator2}
\underline{T}\,\underline{\zeta} = 
\underline{B}_{X\rightarrow R} \, 
\left(1 - \underline K_{\Omega} \right) \,
\underline{A}^{-1} \underline{B}_{S\rightarrow X} \, \underline{\zeta}.
\end{eqnarray}
Finally, we denote by $\underline M_S$ the inner product matrix of $S$ and by $\underline M_R$ the inner product matrix of $R$. Then, the FE approximation of the transfer eigenvalue problem reads as follows: Find the eigenvectors $\underline{\zeta}_{j} \in \mathbb{R}^{N_{S}}$ and the eigenvalues $\lambda_{j} \in \mathbb{R}^{+}_0$ such that
\begin{equation}\label{eq:matrix_version_transfer_eigenvalue_problem}
\underline{T}^{t}\underline M_{R}\underline{T} \,\underline{\zeta}_{j} = \lambda_{j}\, \underline{M}_{S} \,\underline{\zeta}_{j}.
\end{equation}
The coefficients of the FE approximation of the basis functions $\{\phi_{1}^{sp},...,\phi_{n}^{sp}\}$ of the optimal local approximation space 
\begin{equation}\label{eq:discrete optimal space}
R^{n}:=\spanlin \{\phi_{1}^{sp},...,\phi_{n}^{sp}\}
\end{equation}
are then given by
$
\underline{\phi}_{j}^{sp} = \underline{T}\,\underline{\zeta}_{j},$ $j =1,\hdots,n.
$
Adding the representation of the right-hand side, the boundary conditions, and a basis of $\ker(\mathcal{A})$ yields the optimal space $R^{n}_{data,\ker}$. 
 
Note that we may also perform a singular value decomposition of the operator $T$, which reads 
\begin{equation}
T\zeta=\sum^{\min\{N_{S},N_{R}\}}_{j}\sigma_{j}\hat{\phi}_{j}^{sp}(\chi_{j},\zeta)_{S} \quad \text{for} \enspace \zeta \in S,
\end{equation}
with orthonormal bases $\hat{\phi}_{j}^{sp} \in R$, $\chi_{j} \in S$, and singular values $\sigma_{j} \in \mathbb{R}^{+}_{0}$, and define $R^{n}:=\spanlin \{\hat{\phi}_{1}^{sp},...,\hat{\phi}_{n}^{sp}\}$. Up to numerical errors this definition is equivalent to the definition in \cref{eq:discrete optimal space} and there holds $\sigma_{j} = \sqrt{\lambda_{j}}$, $j=1,\hdots, \min\{N_{S},N_{R}\}$, where $\lambda_{j}$ are the eigenvalues of the discrete transfer eigenproblem \cref{eq:matrix_version_transfer_eigenvalue_problem}. Note however that there holds $(\phi_{i}^{sp},\phi_{j}^{sp})_{R} = \delta_{ij} \lambda_{j}$ in contrast to $(\hat{\phi}_{i}^{sp},\hat{\phi}_{j}^{sp})_{R} = \delta_{ij}$.

Finally, we introduce Ritz isomorphisms $D_S: S \rightarrow \mathbb{R}^{N_{S}}$ and $D_R: R \rightarrow \mathbb{R}^{N_{R}}$ which map elements from $S$ or $R$ to a vector containing their FE coefficients in $\R^{N_S}$ or $\R^{N_R}$, respectively. For instance, $D_{S}$ maps a function $\xi = \sum_{i=1}^{N_{S}}\underline{\xi}_{i} \psi_{i}|_{\Gamma_{out}} \in S$ to $\underline{\xi} \in \mathbb{R}^{N_{S}}$. As a result we have the matrix of the transfer operator as $\underline{T} = D_{R} T D_{S}^{-1}$.

\section{Approximating the range of an operator by random sampling}
\label{sect:randomized_la}
In this section we present and analyze an algorithm which constructs a reduced space $R^{n}$ that approximates the range of a finite dimensional linear operator $T$ of rank $N_T$ by iteratively enhancing the reduced space with applications of $T$ to a random function. Although having the transfer operator \cref{eq: discrete transfer operator} in mind we consider the general setting of a finite dimensional linear operator mapping between two finite dimensional Hilbert spaces $S$ and $R$. Note that in the context of localized MOR for inhomogeneous problems it is necessary to enhance $R^{n}$ by the representation of the right-hand side and the boundary conditions.

The algorithm and parts of its analysis are an extension of results in randomized LA \cite{halko2011finding} to the setting of finite dimensional linear operators. In detail we first present an adaptive range finder algorithm in \cref{sec:adaptive range finder algorithm} and discuss its computational complexity. This algorithm relies on a probabilistic a posteriori bound, which is a extension of a result in \cite{halko2011finding} and for which we prove as one new contribution its efficiency in \cref{sec:convergence_criterion}. Starting from results in randomized LA \cite{halko2011finding} we prove in \cref{sec:convergence_rate} that the reduced space $R^{n}$ generated by the algorithm as presented in \cref{sec:adaptive range finder algorithm} yields an approximation that converges with a nearly optimal rate.

\subsection{An adaptive randomized range finder algorithm}\label{sec:adaptive range finder algorithm}
\begin{algorithm2e}
\DontPrintSemicolon
\SetAlgoVlined
\SetKwFunction{AdaptiveRandomizedRangeApproximation}{AdaptiveRandomizedRangeApproximation}
\SetKwInOut{Input}{Input}
\SetKwInOut{Output}{Output}
\Fn{\AdaptiveRandomizedRangeApproximation{$T, \algotol, n_t, \varepsilon_\mathrm{algofail}$}}{
  \Input{Operator $T$,\\target accuracy $\algotol$,\\number of test vectors $n_t$,\\maximum failure probability $\varepsilon_\mathrm{algofail}$}
  \Output{space $R^{n}$ with property $P\left(\norm{T - P_{R^{n}}T} \leq \algotol \right) > \left(1 - \varepsilon_\mathrm{algofail}\right)$}
  \tcc{initialize basis} 
  $B \leftarrow \emptyset$ \label{algo:line:basis_init}\;
  \tcc{initialize test vectors} 
  $M \leftarrow \{T D_S^{-1} \underline r_1, \ \dots, \  T D_S^{-1} \underline r_{n_t} \} $ \label{algo:line:test_vec_init}\;
  \tcc{determine error estimator factor}
  $\varepsilon_\mathrm{testfail} \leftarrow \varepsilon_\mathrm{algofail} / N_T$ \label{algo:line:varepsilon_init} \;
  $c_\mathrm{est} \leftarrow \left[ \sqrt{2 \lambda^{\underline M_S}_{min}} \ \mathrm{erf}^{-1} \left( \sqrt[n_t]{\varepsilon_{\mathrm{testfail}} } \right)
\right]^{-1}$ \label{algo:line:constant_init} \;
  \tcc{basis generation loop}
  \While{$\left( \max_{t \in M} \norm{t}_R \right) \cdot c_\mathrm{est} > \algotol$ \label{algo:line:convergence}}{
    $B \leftarrow B \cup (T D_S^{-1} \underline r) $ \label{algo:line:basis_extension}\;
    $B \leftarrow \mathrm{orthonormalize}(B)$ \label{algo:line:basis_extension2}\;
    \tcc{orthogonalize test vectors to $\mathrm{span}(B)$}
    $M \leftarrow \left\{t - P_{\mathrm{span}(B)} t \ \Big| \ t \in M  \right\}$ \label{algo:line:test_vector_update}\;
  }
  \Return $R^{n} = \mathrm{span}(B)$\;
}
\caption{Adaptive Randomized Range Approximation}
\label{algo:adaptive_range_approximation}
\end{algorithm2e}
We propose an adaptive randomized range approximation algorithm that constructs an approximation space $R^n$
by iteratively extending its basis until a convergence criterion is satisfied.
In each iteration, the basis is extended by the operator $T$ applied to
a random function.

The full algorithm is given in Algorithm \ref{algo:adaptive_range_approximation} and has four input parameters, starting with the operator $T$, whose range should be approximated. This could be represented by a matrix, 
but in the intended context it is usually an implicitly defined operator
which is computationally expensive to evaluate.
Only the evaluation of the operator on a vector is required.
The second input parameter is the target accuracy $\algotol$ such that $\norm{T - P_{R^n}T} \leq \algotol$.
The third input parameter is the number of test vectors $n_t$
to be used in the a posteriori error estimator which we will discuss shortly. 
A typical $n_t$ could be 5, 10, or 20.
The fourth input parameter is the maximum failure probability $\varepsilon_\mathrm{algofail}$ and
the algorithm returns a space which has the required approximation
properties with a probability greater than $1- \varepsilon_\mathrm{algofail}$.

The basis $B$ of $R^n$ is initialized as empty in line \ref{algo:line:basis_init},
test vectors are initialized as the operator applied to random normal vectors in 
line \ref{algo:line:test_vec_init}.
Recall that $T D_S^{-1} \underline r$ is the operator $T$ applied to a random normal vector.
We use the term ``random normal vector'' to denote a vector whose entries are independent and identically distributed random variables
with normal distribution. The main loop of the algorithm is terminated when the following a posteriori norm estimator
applied to $T - P_{R^n}T$ is smaller than $\algotol$.
\begin{definition}[A probabilistic a posteriori norm estimator]\label{def:norm estimator}
To estimate the operator norm of an operator $O: S \rightarrow R$ of rank $N_O$, we define the a posteriori norm estimator $\Delta(O, n_t, \varepsilon_\mathrm{testfail}) $ for $n_t$ test vectors as
\begin{equation}\label{eq:a posteriori error estimator}
\Delta(O, n_t, \varepsilon_\mathrm{testfail}) 
:= \cest \max_{i \in 1, \dots, n_t} \norm{ O \ D_S^{-1} \ \underline r_i }_R.
\end{equation}
Here, $\cest$ is defined as
$
\cest := 
1 / [ \sqrt{2 \lambda^{\underline M_S}_{min}} \ \mathrm{erf}^{-1} ( \sqrt[n_t]{\varepsilon_{\mathrm{testfail}} } ) ]
,
$
$\underline r_i$ are random normal vectors,
and $\lambda^{\underline M_S}_{min}$ is the smallest eigenvalue of the matrix of the inner product in $S$.
\label{definition:test}
\end{definition}
This error estimator $\Delta(O, n_t, \varepsilon_\mathrm{testfail})$ is analyzed in detail in \cref{sec:convergence_criterion}.
The constant $\cest$, which appears in the error estimator,
is calculated in line \ref{algo:line:varepsilon_init} and \ref{algo:line:constant_init} using
$N_T$ --- the rank of operator $T$. In practice $N_{T}$ is unknown
and an upper bound for $N_T$ such as $\min(N_S, N_R)$ can be used instead.
In line \ref{algo:line:convergence} the algorithm assesses if the convergence criterion is already satisfied. Note that the term
$
\left( \max_{t \in M} \norm{t}_R \right) \cdot \cest
$
is the norm estimator \cref{eq:a posteriori error estimator} applied to $T- P_{R^n}T$.
The test vectors are reused for all iterations.
The main loop of the algorithm consists of two parts.
First, the basis is extended in line \ref{algo:line:basis_extension}
and \ref{algo:line:basis_extension2} by applying the operator $T$ to a random normal vector and adding the result to the basis $B$. Then the basis $B$ is orthonormalized. The resulting basis vectors are denoted by $\phi_i^{rnd}$.
We emphasize that the orthonormalization is numerically challenging, 
as the basis functions are nearly linear dependent when
$R^n$ is already a good approximation of the range of $T$.
In the numerical experiments we use the numerically stable Gram-Schmidt with re-iteration from \cite{BEOR14a},
which always succeeded to obtain an orthogonal set of vectors.
Instead of the Gram-Schmidt orthonormalization, one could apply an SVD to the matrix that contains the vectors in $B$ as columns after termination of Algorithm \ref{algo:adaptive_range_approximation} to remove linear dependent vectors.
In the case of almost linear dependent vectors, this could lead to slightly smaller basis sizes.
Note that as we suggest to only remove the linear dependent vectors with the SVD the accuracy of the approximation is not compromised.
Finally, the test vectors are updated in 
line \ref{algo:line:test_vector_update}.

In Algorithm \ref{algo:adaptive_range_approximation},
the smallest eigenvalue of matrix of the inner product in $S$, 
$\lambda^{\underline M_S}_{min}$, or at least a lower bound for it,
is required.
The orthonormalization of $B$ in line \ref{algo:line:basis_extension2}
and the update of test vectors in 
line \ref{algo:line:test_vector_update} use the inner product in
$R$. These aspects should be taken into account when choosing the inner products in $S$ and $R$.

The presented algorithm has good performance properties for operators $T$ which
are expensive to evaluate. To produce the space $R^n$ of dimension $n$, 
it evaluates the operator $n$ times to generate the basis and $n_t$
times to generate the test vectors, so in total $n + n_t$ times.
In contrast, direct calculation of the optimal space, solving the eigenvalue 
problem \cref{eq:transfer eigenvalue problem}, would require
$N_S$ evaluations of the operator and solving a dense eigenproblem
of dimension $N_S \times N_S$.
Exploiting the low rank structure of $T$, one
could calculate the eigenvectors of $T^*T$ using
a Lanczos type algorithm as implemented in ARPACK \cite{Lehoucq1998}, but this would require $\mathcal{O}(n)$ evaluations of
$T$ and $T^*$ in every iteration, 
potentially summing up to much more than $n + n_t$ evaluations,
where the number of iterations is often not foreseeable.

\subsection{A probabilistic a priori error bound}
\label{sec:convergence_rate}
In this subsection we analyze the convergence behavior
of Algorithm \ref{algo:adaptive_range_approximation}. In detail, we derive a probabilistic a priori error bound for the projection error $\norm{T-P_{R^{n}}T}$ and its expected value. Recalling that the optimal convergence rate achieved by the optimal spaces from \cref{theorem:eigenvalue problem} is $\sqrt{\lambda_{n+1}}=\sigma_{n+1}$ we show that the reduced spaces constructed with Algorithm \ref{algo:adaptive_range_approximation} yield an approximation that converges with a nearly optimal rate: 
\begin{proposition}\label{prop:a priori}
\label{thm:convergence_rate_main}
Let $\lambda^{\underline M_S}_{max}$,
$\lambda^{\underline M_S}_{min}$,
$\lambda^{\underline M_R}_{max}$,
and
$\lambda^{\underline M_R}_{min}$ denote the largest and smallest eigenvalues
of the inner product matrices
$\underline M_S$ and
$\underline M_R$, respectively and let $R^{n}$ be the outcome of Algorithm \ref{algo:adaptive_range_approximation}. Then, for $n\geq 4$ there holds
\begin{equation}\label{eq:a priori mean}
\mathbb{E}
\norm{T - P_{R^{n}} T}
\leq
\sqrt{\frac{\lambda^{\underline{M}_{R}}_{max}}{\lambda^{\underline{M}_{R}}_{min}}
\frac{\lambda^{\underline{M}_{S}}_{max}}{\lambda^{\underline{M}_{S}}_{min}}}
\min_{\overset{k+p=n}{k\geq 2, p\geq 2}}
\left[
\left( 1 + \sqrt{  \frac{k}{p-1}  } \right) \sigma_{k+1} +
\frac{ e \sqrt{n}}{p} \left( \sum_{j > k} \sigma^2_j \right) ^{\frac{1}{2}}
\right]
.
\end{equation}

\end{proposition}
Before addressing the proof of \cref{thm:convergence_rate_main} we highlight that for operators with a fast decaying spectrum such as the transfer operator the last term in \cref{eq:a priori mean} behaves roughly as $(e\sqrt{k+p}\sigma_{k+1})/p$ and we therefore obtain an approximation that converges approximately as $\sqrt{n}\sigma_{n+1}$ and thus with a nearly optimal rate. 
\cref{thm:convergence_rate_main} extends the results in Theorem 10.6
in \cite{halko2011finding} to the case of finite dimensional linear operators. The terms consisting of the square root of the conditions of the inner product matrices $\underline{M}_{S}$ and $\underline{M}_{R}$ in \cref{eq:a priori mean}
are due to our generalization from the spectral matrix norm as considered in \cite{halko2011finding} to inner products associated with finite dimensional Hilbert spaces. We present a reformulation in the supplementary materials Proposition SM4.2 where the condition of $M_S$ does not appear. The occurrence of the remaining terms in \cref{eq:a priori mean}
is discussed in section SM3 where we summarize the proof of Theorem 10.6
in \cite{halko2011finding}, which read as follows:
\begin{theorem}\cite[Theorem 10.6]{halko2011finding}
\label{thm:halko106108}
Let $\underline{T}\in \mathbb{R}^{N_{R}\times N_{S}}$ and $\underline P_{R^n, 2}$ be the matrix of the orthogonal projection on $R^n$ in the 
euclidean inner product in $\R^{N_R}$ and $\norm{\cdot}_2$ denote the spectral matrix norm.
Then for $n\geq 4$ it holds
\begin{equation*}
\mathbb{E}
\left( \norm{\underline{T} - \underline{P}_{R^{n},2}\underline{T}}_2 \right)
\leq
\min_{\overset{k+p=n}{k\geq 2, p\geq 2}}
\left[
\left( 1 + \sqrt{  \frac{k}{p-1}  } \right) \underline \sigma_{k+1} +
\frac{ e \sqrt{n}}{p} \left( \sum_{j > k} \underline \sigma^2_j \right) ^{\frac{1}{2}}
\right].
\end{equation*}
\end{theorem}

To proceed with the proof of \cref{thm:convergence_rate_main},
we next bound $\norm{T-P_{R^{n}}T}$ by $\norm{\underline T - \underline P_{R^{n},2}\underline T}_2$ times other terms in \cref{thm:tomatrix}. Then we apply 
\cref{thm:halko106108} to the matrix representation $\underline{T}$ of the operator $T$ and finally bound the singular values $\underline \sigma_i$ of the matrix $\underline T$ by the singular values $\sigma_i$ of the operator $T$ in \cref{thm:svals} below to conclude. 

\begin{lemma}
\label{thm:tomatrix}
There holds for some given reduced space $R^{n}$
\begin{equation*}
\norm{T-P_{R^{n}}T} = 
\sup_{\xi \in S} \inf_{\zeta \in R^{n}} \frac{\| T\xi - \zeta \|_{R}}{\|\xi \|_{S}} 
\leq \sqrt{\frac{\lambda^{\underline{M}_{R}}_{max}}{\lambda^{\underline{M}_{S}}_{min}}} \| \underline{T} - \underline{P}_{R^{n},2}\underline{T} \|_2
.
\end{equation*}
\end{lemma} 
\begin{proof}
\begin{align*}
\sup_{\xi \in S} \inf_{\zeta \in R^{n}} \frac{\| T\xi - \zeta \|_{R}}{\|\xi \|_{S}} &= \sup_{\xi \in S} \frac{\| T\xi - P_{R^{n}}T\xi \|_{R}}{\|\xi \|_{S}} \\
& = \sup_{\underline{\xi} \in \mathbb{R}^{N_{S}}} \frac{\left((\underline{T}\underline{\xi} - \underline{P}_{R^{n}}\underline{T}\underline{\xi})^{T}\underline{M}_{R}(\underline{T}\underline{\xi} - \underline{P}_{R^{n}}\underline{T}\underline{\xi})\right)^{1/2}}{\sqrt{\underline \xi^{T}\underline{M}_{S}\underline{\xi}}} \\
& \leq \sup_{\underline{\xi} \in \mathbb{R}^{N_{S}}} \frac{\left((\underline{T}\underline{\xi} - \underline{P}_{R^{n},2}\underline{T}\underline{\xi})^{T}\underline{M}_{R}(\underline{T}\underline{\xi} - \underline{P}_{R^{n},2}\underline{T}\underline{\xi})\right)^{1/2}}{\sqrt{\underline \xi^{T}\underline{M}_{S}\underline{\xi}}} \\
& \leq \sqrt{\frac{\lambda^{\underline{M}_{R}}_{max}}{\lambda^{\underline{M}_{S}}_{min}}} \sup_{\underline{\xi} \in \mathbb{R}^{N_{S}}} \frac{\| \underline{T}\underline{\xi} - \underline{P}_{R^{n},2}\underline{T}\underline{\xi} \|_{2} }{\|\underline{\xi}\|_{2}}
\end{align*}
\end{proof}

\begin{lemma}
\label{thm:svals}
Let the singular values $\underline{\sigma}_{j}$ of the matrix $\underline T$ be sorted in non-increasing order, i.e. $\underline{\sigma}_{1} \geq \hdots \geq \underline{\sigma}_{N_{R}}$ and $\sigma_{j}$ be the singular values of the operator $T$, also sorted non-increasing. Then there holds 
$
\underline{\sigma}_{j} \leq
(\lambda^{\underline{M}_{S}}_{max}/\lambda^{\underline{M}_{R}}_{min})^{1/2}
\sigma_{j}$ for all $j=1,\hdots,N_T$.
\end{lemma}
\begin{proof}
For notational convenience we denote within this proof the $j$-th eigenvalue of a matrix 
$\underline A$ by $\lambda_{j}(\underline A)$. All singular values for $j = 1, \dots, N_T$ are different from zero. Therefore, there holds
$\underline \sigma_j^2 = \lambda_j(\underline T^t \underline T)$
and
$\sigma_j^2 = \lambda_j(\underline M_S^{-1} \underline T^t \underline M_R \underline T)$
.
Recall that $\underline T$ is the matrix representation of $T$ and note that 
$\underline M_S^{-1} \underline T^t \underline M_R$
is the matrix representation of the adjoint operator $T^*$. The non-zero eigenvalues of a product of matrices $\underline A \underline B$ 
are identical to the non-zero eigenvalues of the product $\underline B \underline A$
(see e.g.~\cite[Theorem 1.3.22]{horn2012matrix}), hence
$
\lambda_j(\underline M_S^{-1} \underline T^t \underline M_R \underline T)
=
\lambda_j(
\underline T^t \underline M_R \underline T \underline M_S^{-1} )
.
$
We may then apply the Courant minimax principle to infer 
$
\lambda_{j}(\underline{T}^{t}\underline{M}_{R}\underline{T}) (\lambda^{\underline{M}_{S}}_{max})^{-1} \leq \lambda_{j}(\underline{M}_{S}^{-1}\underline{T}^{t}\underline{M}_{R}\underline{T}).
$
Employing once again cyclic permutation and the Courant minimax principle yields
\begin{equation}\label{est2:proof eigenvalue est}
\lambda_{j}(\underline{T}^{t}\underline{T}) 
\leq \lambda_{j}(\underline{T}^{t}\underline{M}_{R}\underline{T}) \frac{1}{\lambda^{\underline{M}_{R}}_{min}} \leq \lambda_{j}(\underline{M}_{S}^{-1}\underline{T}^{t}\underline{M}_{R}\underline{T})\frac{\lambda^{\underline{M}_{S}}_{max}}{\lambda^{\underline{M}_{R}}_{min}}
\end{equation}
and thus the claim.
\end{proof}

\begin{remark}
The result of Algorithm \ref{algo:adaptive_range_approximation},
when interpreted as functions and not as coefficient vectors,
is independent of the choice of the basis in $R$.
Disregarding numerical errors,
the result would be the same if the algorithm was executed in an orthonormal basis in $R$.
Thus, we would expect \cref{thm:convergence_rate_main} to hold
also without the factor
$(\lambda^{\underline{M}_{R}}_{max}/\lambda^{\underline{M}_{R}}_{min})^{1/2}$.
\end{remark}

\subsection{Adaptive convergence criterion and a probabilistic a posteriori error bound}
\label{sec:convergence_criterion}
When approximating the range of an operator, usually its singular values are
unknown. To construct a space with prescribed approximation
quality, Algorithm \ref{algo:adaptive_range_approximation} uses the probabilistic a posteriori error estimator defined in \cref{definition:test}, which is analyzed in this subsection.
\begin{proposition}[Norm estimator failure probability]
\label{thm:errest_reliable}
The norm estimator 
$
\Delta(O, n_t, \varepsilon_\mathrm{testfail})
$
is an upper bound of the operator norm $\norm{O}$ with probability greater or equal than $(1 - \varepsilon_\mathrm{testfail})$.
\end{proposition}

\begin{proposition}[Norm estimator effectivity]
\label{thm:errest_efficient}
Let the effectivity $\eta$ of the norm estimator $
\Delta(O, n_t, \varepsilon_\mathrm{testfail})
$ be defined as
\begin{equation}\label{eq:effectivity}
\eta(O, n_t, \varepsilon_\mathrm{testfail}) := 
\frac{
\Delta(O, n_t, \varepsilon_\mathrm{testfail}) 
}{
\norm{O}
}.
\end{equation}
Then, there holds
$$
P\Big( \eta \leq \ceff \Big) \geq 1 - \varepsilon_\mathrm{testfail}
,
$$
where the constant $\ceff$ is defined as
$$
\ceff := 
\left[
Q^{-1}\left(\frac{N_O}{2}, \frac{\varepsilon_\mathrm{testfail}}{n_t}\right)
\frac{\lambda^{\underline M_S}_{max}}{\lambda^{\underline M_S}_{min}}
\left( \mathrm{erf}^{-1} \left( \sqrt[n_t]{\varepsilon_\mathrm{testfail}} \right) \right)^{-2}
\right]^{1/2}
$$
and $Q^{-1}$ is the inverse of the upper normalized incomplete gamma function;
that means $Q^{-1}(a,y) = x$ when $Q(a,x) = y$.\footnote{
Recall that the definition of the upper normalized incomplete gamma function is
$$
Q(a,x) = 
\dfrac{\int_x^\infty t^{a-1} e^{-t} \mathrm{d} t
}{
\int_0^\infty t^{a-1} e^{-t} \mathrm{d} t
}.
$$
}
\end{proposition}

The proofs of \cref{thm:errest_reliable,thm:errest_efficient} follow at the end of this subsection.

In \cref{thm:errest_reliable} we analyzed the probability for one estimate to fail. Based on that, we can analyze the algorithm failure probability. To quantify this probability, we first note that Algorithm \ref{algo:adaptive_range_approximation} will terminate after at most $N_T$ steps. Then, the approximation space $R^{n}$
has the same dimension as $\range(T)$ and as $R^{n}\subset \range(T)$ we have $R^{n}=\range(T)$ and thus $\norm{T- P_{R^{n}}T}=0$. The a posteriori error estimator defined in \cref{definition:test} is therefore executed at most $N_T$ times. Each time, the probability for failure is given by \cref{thm:errest_reliable} and with a union bound argument we may then infer that the failure probability for the whole algorithm is
$
\varepsilon_\mathrm{algofail} 
\leq
N_T
 \ \varepsilon_\mathrm{testfail}
.
$

To prove \cref{thm:errest_reliable,thm:errest_efficient},
it is central to analyze the distribution of the inner product $(v, D_S^{-1}\underline r)_S$ for any $v \in S$ with $\norm{v}_S = 1$ and a random normal vector $\underline{r}$.

\begin{lemma}[Distribution of inner product]
The inner product of a normed vector $v$ in $S$ with a random normal vector
$(v, D_S^{-1}\underline r)_S$ 
is a Gaussian distributed random variable with mean zero and variance $s^2$, where
$
\lambda^{\underline M_S}_{min} \leq s^2 \leq \lambda^{\underline M_S}_{max}.
$
\end{lemma}

\begin{proof}
We use the spectral decomposition of the inner product matrix \\
$\underline M_S = \sum_{i=1}^{N_S} \underline m_{S,i} \lambda^{\underline M_S}_i \underline m_{S,i}^T$
with eigenvalues $\lambda^{\underline M_S}_i$ and eigenvectors $\underline m_{S,i}$.
There holds
\begin{align}
(v, D_S^{-1} \underline r)_S
&= \sum_{i=1}^{N_S} (D_S v)^T \underline m_{S,i} \lambda^{\underline M_S}_i \underline m_{S,i}^T \underline r . %
\end{align}
As $m_{S,i}$ is normed with respect to the euclidean inner product, the
term $\underline m_{S,i}^T \underline r$ is a normal distributed random variable.
Using the rules for addition and scalar multiplication of Gaussian random variables,
one sees that the inner product $(v, D_S^{-1}\underline r)_S$ is a Gaussian random variable with
variance
$
s^2 = \sum_{i=1}^{N_S} ( (D_S v)^T \underline m_{S,i} \lambda^{\underline M_S}_i  )^2
$
The variance $s^2$ can easily be bounded as follows:
\begin{eqnarray*}
s^2 
&= \sum_{i=1}^{N_S} \left( (D_S v) ^T \underline m_{S,i} \lambda^{\underline M_S}_i  \right) ^2 
\leq \sum_{i=1}^{N_S} \left( (D_S v) ^T \underline m_{S,i} \right) ^2 \lambda^{\underline M_S}_i \max_i (\lambda^{\underline M_S}_i)
&= \lambda^{\underline M_S}_{max}\\
s^2 
&= \sum_{i=1}^{N_S} \left( (D_S v) ^T \underline m_{S,i} \lambda^{\underline M_S}_i  \right) ^2
\geq \sum_{i=1}^{N_S} \left( (D_S v) ^T \underline m_{S,i} \right) ^2 \lambda^{\underline M_S}_i \min_i (\lambda^{\underline M_S}_i)
&= \lambda^{\underline M_S}_{min}
\end{eqnarray*}
\end{proof}\\
Using this result, we can prove \cref{thm:errest_reliable,thm:errest_efficient}.
Proof of \cref{thm:errest_reliable}:
\begin{proof}
We analyze the probability for the event that the norm estimator {\color{white}sfasdfsfd}
$
\Delta(O, n_t, \varepsilon_\mathrm{testfail})
$
is smaller than the operator norm $\norm{O}$:
\begin{eqnarray*}
P\Big( \Delta(O, n_t, \varepsilon_\mathrm{testfail}) < \norm{O} \Big)
&=&
P\Big( 
\cest
\max_{i \in 1, \dots, n_t} \norm{ O \ D_S^{-1} \ \underline r_i }_R 
< \norm{O}
\Big)
.
\end{eqnarray*}
The probability that all test vector norms are smaller than a 
certain value is the 
the product of 
the probabilities 
that each test vector is smaller than that value.
So with a new random normal vector $\underline r$ it holds
\begin{eqnarray*}
P\Big( \Delta(O, n_t, \varepsilon_\mathrm{testfail}) < \norm{O} \Big)&=&
P\Big( 
\cest
\norm{ O \ D_S^{-1} \ \underline r }_R 
< \norm{O}
\Big)^{n_t}.
\end{eqnarray*}
Using the singular value decomposition of the operator $O$: $O \varphi = \sum_i u_i \sigma_i (v_i, \varphi)_S$ we obtain
\begin{eqnarray*}
P\Big( \Delta(O, n_t, \varepsilon_\mathrm{testfail}) < \norm{O} \Big)&\leq&
P\Big( 
\cest
\norm{ u_1 \sigma_1 \ \left(v_1, D_S^{-1} \ \underline r\right)_S }_R 
< \norm{O}
\Big)^{n_t}\\
&=&
P\Big( 
\cest
\sigma_1 \ \left|\left(v_1, D_S^{-1} \ \underline r\right)_S\right|
< \norm{O}
\Big)^{n_t}\\
&=&
P\Big( 
\cest
\left|\left(v_1, D_S^{-1} \ \underline r\right)_S\right|
< 1
\Big)^{n_t}.\\
\end{eqnarray*}
The inner product $\left|\left(v_1, D_S^{-1} \ \underline r\right)_S\right|$ is a Gaussian distributed 
random variable with variance greater $\lambda^{\underline M_S}_{min}$, so
with a new normal distributed random variable $r'$ it holds
\begin{eqnarray*}
P\Big( \Delta(O, n_t, \varepsilon_\mathrm{testfail}) < \norm{O} \Big)&\leq&
P\Big( 
\sqrt{\lambda^{\underline M_S}_{min}} |r'|
<
\sqrt{2 \lambda^{\underline M_S}_{min}} \cdot \mathrm{erf}^{-1}\left(\sqrt[n_t]{\varepsilon_\mathrm{testfail}}\right)
\Big)^{n_t}\\
&=&
\mathrm{erf} \left( \frac{\sqrt{2} \mathrm{erf}^{-1}\left(\sqrt[n_t]{\varepsilon_\mathrm{testfail}}\right) }{\sqrt{2}} \right)^{n_t}
 = \varepsilon_\mathrm{testfail}.
\end{eqnarray*}
\end{proof}

Proof of \cref{thm:errest_efficient}:
\begin{proof}
The constant $\cest$ is defined as in the proof of \cref{thm:errest_reliable}.
To shorten notation, we write $\cestshort$ for $\cest$ and $\ceffshort$ for $\ceff$ within this proof. Invoking the definition of $\Delta(O, n_t, \varepsilon_\mathrm{testfail})$ yields
\begin{eqnarray*}
P\Big( \Delta(O, n_t, \varepsilon_\mathrm{testfail}) > \ceffshort \norm{O} \Big)
&=&
P\Big( 
\cestshort
\max_{i \in 1, \dots, n_t} \norm{ O \ D_S^{-1} \ \underline r_i }_R 
> \ceffshort \norm{O}
\Big)
\end{eqnarray*}
and by employing a new random normal vector $\underline r$ we obtain
\begin{eqnarray*}
P\Big( \Delta(O, n_t, \varepsilon_\mathrm{testfail}) > \ceffshort \norm{O} \Big)&\leq&
n_t
P\Big( 
\cestshort
\norm{ O \ D_S^{-1} \ \underline r }_R 
> \ceffshort \norm{O}
\Big).
\end{eqnarray*}
Using the singular value decomposition of the operator $O$: $O \varphi = \sum_i u_i \sigma_i (v_i, \varphi)_S$ results in
\begin{eqnarray*}
P\Big( \Delta(O, n_t, \varepsilon_\mathrm{testfail}) > \ceffshort \norm{O} \Big)&\leq&
n_t
P\Big( 
\cestshort
\norm{\sum_i u_i \sigma_1 (v_i, D_S^{-1} \underline r)_S}_R
> \ceffshort \norm{O}
\Big).
\end{eqnarray*}
For a new random normal variables $r_i$ we have
\begin{eqnarray*}
P\Big( \Delta(O, n_t, \varepsilon_\mathrm{testfail}) > \ceffshort \norm{O} \Big)&\leq&
n_t
P\Big( 
\cestshort
\sigma_1 \sqrt{\lambda^{\underline M_S}_{max}} \sqrt{\sum_i r_i^2}
> \ceffshort \norm{O}
\Big)\\
=
n_t
P\Big(
\sqrt{\sum_i r_i^2}
> \frac{\ceffshort}{\cestshort} \sigma_1^{-1}
\sqrt{\lambda^{\underline M_S}_{max}}^{-1}
\norm{O}
\Big)
&=&
n_t
P\Big(
\sqrt{\sum_i r_i^2}
> \frac{\ceffshort}{\cestshort}
\sqrt{\lambda^{\underline M_S}_{max}}^{-1}
\Big)\\
&=&
n_t
P\Big(
\sum_i r_i^2
> \frac{\ceffshort^2}{\cestshort^2}
\sqrt{\lambda^{\underline M_S}_{max}}^{-2}
\Big).
\end{eqnarray*}
The sum of squared random normal variables is a random variable with chi-squared distribution.
Its cumulative distribution function is the incomplete, normed gamma function.
As we have a $>$ relation, the upper incomplete normed gamma function is used,
which we denote by $Q(\frac k 2, \frac x 2)$ here.
Therefore, we conclude
\begin{eqnarray*}
P\Big( \Delta(O, n_t, \varepsilon_\mathrm{testfail}) > \ceff \norm{O} \Big)&\leq&
n_t Q\left(\frac{N_O}{2}, \frac{\ceff^2}{\cest^2} \frac{1}{2 \lambda^{\underline M_S}_{max}}\right)\\
&=& \varepsilon_\mathrm{testfail}.
\end{eqnarray*}
\end{proof}

\section{Numerical experiments}\label{sect:numerical_experiments}

In this section we demonstrate first that the reduced local spaces generated by Algorithm \ref{algo:adaptive_range_approximation} yield an approximation that converges at a nearly optimal rate. Moreover, we validate the a priori error bound in \cref{eq:a priori mean}, the a posteriori error estimator \cref{eq:a posteriori error estimator}, and the effectivity \cref{eq:effectivity}.
To this end, we consider four test cases, starting in \cref{subsect:analytic} with an
example for which the singular values of the transfer operator are known.
The main focus of this subsection is a thorough validation of the theoretical findings in \cref{sect:randomized_la}, including a comprehensive testing on how the results depend on various parameters such as the basis size $n$, the number of test vectors $n_{t}$, and the mesh size.
In addition, CPU time measurements are given.
The second numerical example in \cref{subsect:helmholtz}
examines the behavior of the proposed algorithm in the more challenging case of the Helmholtz equation.
In \cref{subsect:linearelasticity} we numerically analyze the theoretical results from \cref{sect:randomized_la} for a transfer operator whose singular values decay rather slowly and discrete spaces with large $N$, $N_{S}$, and $N_{R}$. Furthermore, we demonstrate that Algorithm \ref{algo:adaptive_range_approximation} is computationally efficient. Finally, we employ the GFEM to construct a global approximation from the local reduced spaces generated by Algorithm \ref{algo:adaptive_range_approximation} in the fourth test case in \cref{subsect:GFEM}, demonstrating that the excellent local approximation capacities of the local reduced spaces carry over to the global approximation.

For the implementation of the first test case, no FEM software library was used.
The implementation for the third test case is based on the finite element library \texttt{libMesh} \cite{Kiretal06}.
For the second and fourth test case we used the software library \texttt{pyMOR} \cite{MiRaSc16}.
The complete source code for reproduction of all results shown in \cref{subsect:analytic,subsect:helmholtz,subsect:GFEM} is provided in
\cite{andreas_buhr_2018}.

\subsection{Analytic interface problem}\label{subsect:analytic}
To analyze the behavior of the proposed algorithm, we first
apply it to an analytic problem where the singular values of the transfer operator are known. We refer to this numerical example as {\it Example 1}. We consider the problem $\mathcal{A}=-\Delta$, $f =0$, and assume that $\Omega = (-L,L)\times (0,\analyticwidth)$, $\Gamma_{out}=\{-L,L\}\times (0,\analyticwidth)$, and $\Gamma_{in} = \{0\}\times (0,\analyticwidth)$. Moreover, we prescribe homogeneous Neumann boundary conditions on $\partial \Omega \setminus \Gamma_{out}$ and arbitrary Dirichlet boundary conditions on $\Gamma_{out}$,
see also \cref{fig:illustration geometry} (left).
The analytic solution is further discussed in the supplementary materials section SM2. This example was introduced in \cite[Remark 3.3]{SmePat16}. We equip $S$ and $R$ with the $L^2$-inner product on the respective interfaces. Recall that the transfer operator maps the Dirichlet data to the inner interface,
i.e.~with
$H$ as the space of all discrete solutions, we define
\begin{equation}\label{eq:mod transfer operator}
T(v|_{\Gamma_{out}}) := v|_{\Gamma_{in}} \qquad \forall v \in H.
\end{equation}\footnote{Thanks to the form of the $A$-harmonic functions (SM2.1) and the fact that $\Omega$ is symmetric with respect to the $x_{2}$-axis, we have that $u - (1/|\Gamma_{out}|) \int_{\Gamma_{out}} u = u - (1/|\Omega|) \int_{\Omega} u$ and therefore that the singular vectors and singular values of \cref{eq:transfer_op_gamma} equal the ones of \cref{eq:mod transfer operator} apart from the constant function, which has to be added for the former but not for the latter.}
The singular values of the transfer operator are
$
\sigma_{i}=1/ \left(\sqrt{2}\cosh((i-1)\pi L / \analyticwidth)\right).
$ 

For the experiments, we use $L = \analyticwidth = 1$, unless stated otherwise. We discretize the problem by meshing it with a regular mesh of squares of size $h \cdot h$, 
where $1/h$ ranges from 20 to 320 in the experiments. On each square, we use bilinear 
Q1 ansatz functions,
which results in e.g.~51,681 DOFs, $N_S = 322$ and $N_R= 161$ for $1/h=160$.

In \cref{fig:random modes} the first five basis vectors
as generated by Algorithm \ref{algo:adaptive_range_approximation}
in one particular run
are shown side by side with the first five basis vectors
of the optimal space, i.e.~the optimal modes in
\cref{fig:optimal modes}.
While not identical, the basis functions
generated using the randomized approach are
smooth and have strong similarity with the optimal ones.
Unless stated otherwise, we
present statistics over 100,000 evaluations,
use a maximum failure probability of $\varepsilon_\mathrm{algofail} = 10^{-15}$,
and use $\min(N_S, N_R)$ as an upper bound for $N_T$.

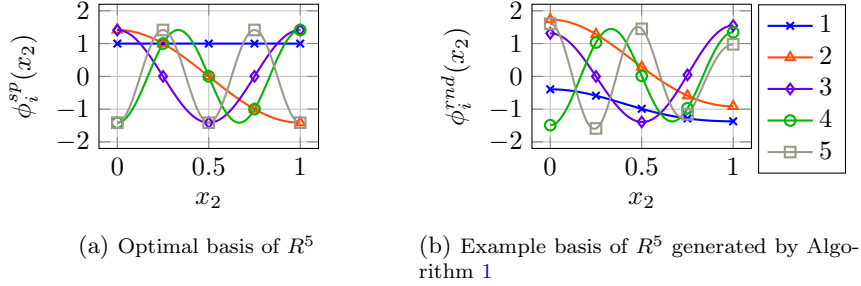
\begin{figure}
\centering
\begin{subfigure}[t]{0.4\textwidth}
\begin{tikzpicture}
\begin{axis}[
    width=4.5cm,
    height=3.5cm,
    xlabel=$x_2$,
    ylabel=$\phi_i^{sp}(x_2)$,
    legend pos=outer north east,
    ymin=-2.2,
    ymax=2.2,
    grid=both,
    grid style={line width=.1pt, draw=gray!20},
    major grid style={line width=.2pt,draw=gray!50},
  ]
  \addplot+[blue,  mark=x,        mark repeat=40, thick] table[x expr=\coordindex / 160, y index=0] {omodes.dat};
  \addplot+[oran,  mark=triangle, mark repeat=40, thick] table[x expr=\coordindex / 160, y index=1] {omodes.dat};
  \addplot+[lila,  mark=diamond,  mark repeat=40, thick] table[x expr=\coordindex / 160, y index=2] {omodes.dat};
  \addplot+[gruen,  mark=o,        mark repeat=40, thick] table[x expr=\coordindex / 160, y index=3] {omodes.dat};
  \addplot+[grau,  mark=square,   mark repeat=40, thick] table[x expr=\coordindex / 160, y index=4] {omodes.dat};
\end{axis}
\end{tikzpicture}
\subcaption{\footnotesize Optimal basis of $R^5$}
\label{fig:optimal modes}
\end{subfigure}\quad
\begin{subfigure}[t]{0.45\textwidth}
\centering
\begin{tikzpicture}
\begin{axis}[
    width=4.5cm,
    height=3.5cm,
    xlabel=$x_2$,
    ylabel=$\phi_i^{rnd}(x_2)$,
    ymin=-2.2,
    ymax=2.2,
    legend pos=outer north east,
    grid=both,
    grid style={line width=.1pt, draw=gray!20},
    major grid style={line width=.2pt,draw=gray!50},
  ]
  \addplot+[blue,  mark=x,        mark repeat=40, thick] table[x expr=\coordindex / 160, y index=0] {modes.dat};
  \addplot+[oran,  mark=triangle, mark repeat=40, thick] table[x expr=\coordindex / 160, y index=1] {modes.dat};
  \addplot+[lila,  mark=diamond,  mark repeat=40, thick] table[x expr=\coordindex / 160, y index=2] {modes.dat};
  \addplot+[gruen,  mark=o,        mark repeat=40, thick] table[x expr=\coordindex / 160, y index=3] {modes.dat};
  \addplot+[grau,  mark=square,   mark repeat=40, thick] table[x expr=\coordindex / 160, y index=4] {modes.dat};
  \legend{1,2,3,4,5}
\end{axis}
\end{tikzpicture}
\subcaption{\footnotesize Example basis of $R^5$ generated by Algorithm \ref{algo:adaptive_range_approximation}}
\label{fig:random modes}
\end{subfigure}
\vspace{-10pt}
\caption{\footnotesize
Comparison of optimal basis functions with the basis functions generated by Algorithm
\ref{algo:adaptive_range_approximation} for {\it Example 1}.
Basis functions are normalized to an $L^2(\Gamma_{in})$ norm of one.
}
\label{img:modes}
\vspace{-15pt}
\end{figure}

\begin{figure}
\centering
\begin{subfigure}[t]{0.55\textwidth}
\centering
\begin{tikzpicture}
\begin{semilogyaxis}[
    width=5cm,
    height=5cm,
    xmin=-1,
    xmax=13,
    ymin=1e-17,
    ymax=1e2,
    xlabel=basis size $n$,
    ylabel=$\norm{T - P_{R^{n}}T}$,
    grid=both,
    grid style={line width=.1pt, draw=gray!20},
    major grid style={line width=.2pt,draw=gray!50},
    minor xtick={1,2,3,4,5,6,7,8,9,10,11,12,13,14,15},
    minor ytick={1e-15, 1e-14, 1e-13, 1e-12, 1e-11, 1e-10, 1e-9, 1e-8, 1e-7, 1e-6, 1e-5, 1e-4, 1e-3, 1e-2, 1e-1, 1e0, 1e1, 1e2, 1e3},
    ytick={1e0, 1e-5, 1e-10, 1e-15},
    max space between ticks=25,
    legend style={at={(1.4,1.35)},anchor=north east},
  ]
  \addplot+[densely dotted, blue,  mark=o, mark options={solid}] table[x expr=\coordindex,y index=4] {twosquares_percentiles.dat};
  \addplot+[densely dashed, blue,  mark=x, mark options={solid}] table[x expr=\coordindex,y index=3] {twosquares_percentiles.dat};
  \addplot+[solid,          black, thick, mark=x] table[x expr=\coordindex,y index=2] {twosquares_percentiles.dat};
  \addplot+[densely dashed, gruen, mark=x, mark options={solid}] table[x expr=\coordindex,y index=1] {twosquares_percentiles.dat};
  \addplot+[densely dotted, gruen, mark=o, mark options={solid}] table[x expr=\coordindex,y index=0] {twosquares_percentiles.dat};
  \addplot+[solid,          red,   mark=0, thick] table[x index=0, y index=1] {experiment_twosquares_svddecay_8.dat};
  \legend{max, 75 percentile, 50 percentile, 25 percentile, min, $\sigma_{i+1}$};
\end{semilogyaxis}
\end{tikzpicture}
\subcaption{\footnotesize Percentiles, worst case, and best case}
\label{fig:deviation percentiles}
\end{subfigure}
\begin{subfigure}[t]{0.4\textwidth}
\centering
\begin{tikzpicture}
\begin{semilogyaxis}[
    width=5cm,
    height=5cm,
    xmin=-1,
    xmax=13,
    ymin=1e-17,
    ymax=1e2,
    xlabel=basis size $n$,
    yticklabels={,,},
    grid=both,
    grid style={line width=.1pt, draw=gray!20},
    major grid style={line width=.2pt,draw=gray!50},
    minor xtick={1,2,3,4,5,6,7,8,9,10,11,12,13,14,15},
    minor ytick={1e-16, 1e-15, 1e-14, 1e-13, 1e-12, 1e-11, 1e-10, 1e-9, 1e-8, 1e-7, 1e-6, 1e-5, 1e-4, 1e-3, 1e-2, 1e-1, 1e0, 1e1, 1e2, 1e3},
    ytick={1e0, 1e-5, 1e-10, 1e-15},
    max space between ticks=25,
    legend style={at={(1.1,1.35)},anchor=north east},
  ]
  \addplot+[] table[x index=0, y index=1]{twosquares_expectation.dat};
  \addplot+[solid,          black, thick, mark=x] table[x expr=\coordindex,y index=0] {twosquares_means.dat};
  \legend{a priori limit for mean, mean};
\end{semilogyaxis}
\end{tikzpicture}
\subcaption{\footnotesize Mean of deviation and a priori limit}
\label{fig:deviation mean}
\end{subfigure}
\vspace{-10pt}
\caption{\footnotesize Projection 
error $\sup_{\xi \in S} \inf_{\zeta \in R^{n}} \frac{\| T\xi - \zeta \|_{R}}{\|\xi \|_{S}} = \norm{T - P_{R^{n}}T}$
over basis size $n$ for {\it Example 1}.
Meshsize $h=1/160$.\\
}
\label{fig:interface_a_priori}
\vspace{-15pt}
\end{figure}

We first quantify the approximation quality of the spaces
$R^n$ in dependence of the basis size $n$, disregarding
the adaptive nature of Algorithm \ref{algo:adaptive_range_approximation}.
In \cref{fig:deviation percentiles}, statistics
over the achieved projection error $\norm{T - P_{R^n}T}$
are shown along with the singular values 
$\sigma_{n+1}$ of the transfer operator $T$.
$\sigma_{n+1}$ is a lower bound for the
projection error and it is the projection error
that is achieved using an optimal basis.
It shows that while the algorithm most of the time
produces a basis nearly as good as the optimal basis, sometimes it
needs two or three basis vectors more. 
This is in line with the predictions by theory, see the discussion
after \cref{thm:convergence_rate_main}.
The mean value of the projection error 
converges with the same rate as the a priori error bound
given in
\cref{thm:convergence_rate_main}
with increasing basis size. The a priori error bound is consistently
around three orders of magnitude larger than the actual error, 
until the actual error hits the numerical noise between
$10^{-14}$ and $10^{-15}$, see \cref{fig:deviation mean}. This is mainly due to the fact that the singular values decay very fast for the present example and an index shift in the singular values by $p\geq 2$ as required by the a priori error bound \cref{eq:a priori mean} therefore results in a much smaller error than predicted by the a priori error bound.  Note that we have $(\lambda_{max}^{\underline{M}_{R}}/\lambda_{min}^{\underline{M}_{R}})^{1/2}\approx (\lambda_{max}^{\underline{M}_{S}}/\lambda_{min}^{\underline{M}_{S}})^{1/2}\approx 2$. 

The adaptive behavior of Algorithm \ref{algo:adaptive_range_approximation}
is analyzed in \cref{fig:interface_adaptive_quartiles}.
\cref{fig:adaptive performance} shows that
for $n_t = 10$
the algorithm succeeded to generate a space with the
requested approximation quality every single time in 
the 100,000 test runs and most of the time,
the approximation quality is about one or two
orders of magnitude better than required.
\cref{fig:number of testvecs} shows
the influence of the number of test vectors $n_t$:
With a low number of test vectors like 3 or 5, 
the algorithm produces spaces with an approximation
quality much better than requested, which 
is unfavorable as the basis sizes are larger than 
necessary. 10 or 20 test vectors seem to be a 
good compromise, as enlarging $n_t$
to 40 or 80 results in only little improvements
while increasing computational cost. This different behavior of Algorithm \ref{algo:adaptive_range_approximation} for various numbers of test vectors $n_{t}$ is due to the scaling of the effectivity of the a posteriori error estimator $\eta(T-P_{R^{n}}T, n_t, \varepsilon_\mathrm{testfail})$ as defined in \cref{eq:effectivity} in the number of test vectors $n_{t}$: The median effectivity $\eta(T-P_{R^{n}}T, n_t, \varepsilon_\mathrm{testfail})$
is 29.2 for $n_t = 10$, 10.4 for $n_t = 20$, and 6.1 for $n_t = 40$. We may thus also  conclude that the a posteriori error estimator \cref{eq:a posteriori error estimator} is a sharp bound for the present test case. 

Analyzing the numerical effectivity of the a posteriori error estimator
$\eta(T-P_{R^{n}}T, n_t, \varepsilon_\mathrm{testfail})$ and comparing it to its
theoretical upper bound $\ceff$ in \cref{fig:numerical efficiency 1},
it can be observed that the theoretical upper bound becomes a sharper bound
with increasing number of test vectors $n_t$.
The reason is the decreasing dispersion of the normalized
maximal test vector norm
\begin{equation}
\label{eq:normalized test vector norm}
\frac{\max\limits_{i=1, \dots, n_t} \norm{(T - P_{R^n}T) D_S^{-1} \underline r_i}_R }{ \norm{T - P_{R^n}T}},
\end{equation}
as shown in \cref{fig:numerical efficiency 2}.
The normalized test vector norm is bound from above by $\cest^{-1} \ceff$ and from below by $\cest^{-1}$
with error probability $\varepsilon_\mathrm{testfail}$.

The quality of the produced spaces $R^n$
should be independent of the mesh size $h$.
\cref{fig:h dependency deviation}
confirms this.
After a preasymptotic regime, the deviation $\norm{T - P_{R^n}T}$
is independent of the mesh size.
In the preasymptotic regime, the finite element
space is not capable of approximating the
corresponding modes.
But while the deviation $\norm{T - P_{R^n}T}$
is independent of the mesh size, the norm of the test
vectors used in the a posteriori error estimator
in Algorithm \ref{algo:adaptive_range_approximation}
is not (see \cref{fig:h dependency testvecnorm}).
The maximum norm of test vectors scales with the deviation and with $\sqrt{h}$.
In the adaptive algorithm, the scaling with $\sqrt{h}$ is compensated by the factor 
$({\lambda^{\underline M_S}_{min}})^{-1/2}$ in $\cest$.
To analyze the behavior in $h$, the geometry parameters were chosen as $L=0.5$ and $\analyticwidth=1$
to have a slower decay of the singular values of the transfer operator.

To examine CPU times we use {\it Example 1}
in a larger configuration
with $L=1$, $W=8$ and $1/h = 200$.
This results in 638.799 unknowns,
$N_S = 3202$, and $N_R = 1601$.
The
measured CPU times
for a simple, single threaded implementation
are given in 
\cref{tab:cpu_times}.
The transfer operator is implemented
implicitly. Its matrix is not
assembled. Instead, the corresponding problem
is solved using the sparse direct solver 
SuperLU \cite{superlu_ug99,superlu99}
each
time the operator is applied.
For Algorithm \ref{algo:adaptive_range_approximation},
a target accuracy $\algotol$ of $10^{-4}$,
the number of testvectors $n_t = 20$,
and a maximum failure probability $\varepsilon_\mathrm{algofail} = 10^{-15}$
is used. In one test run, it resulted in an approximation space $R^n$
of dimension $39$. It only evaluated the operator $n + n_t = 59$ times.
Each operator evaluation was measured to take $0.301$ seconds,
so a runtime of approximately $(n + n_t) * 0.301\mathrm{s} \approx 17.8$s is expected.
The measured runtime of 20.4 seconds is slightly higher, due to the
orthonormalization of the basis vectors and the projection of the test vectors.

CPU times for the calculation of the optimal space of same size
are given for comparison. 
The ``eigs'' function in ``scipy.sparse.linalg'', which is based on ARPACK, is used
to find the eigensystem of
$T T^*$.
However, the calculation using ARPACK is not adaptive. 
To employ ARPACK, the required number of vectors
has to be known in advance,
which is why we expect that in general,
the comparison would be even more in favor of the adaptive randomized
algorithm.

\begin{table}
\begin{center}
\begin{tabular}{r|c}
\multicolumn{2}{c}{\emph{Properties of transfer operator}}\\
\hline
unknowns of corresponding problem& 638,799\\
LU factorization time in s & 14.1\\
operator evaluation time in s & 0.301 \\
adjoint operator evaluation time in s & 0.301 
\end{tabular}\\
\vspace{10pt}
\begin{tabular}{r|c|c}
\multicolumn{3}{c}{\emph{Properties of basis generation}}\\
\cline{2-3}
& Algorithm \ref{algo:adaptive_range_approximation} & Scipy/ARPACK \\
\hline
(resulting) basis size $n$ & 39 & 39 \\
operator evaluations & 59 & 79 \\
adjoint operator evaluations & 0 & 79 \\
\hline
execution time in s (w/o factorization) & 20.4 & 47.9 
\end{tabular}
\end{center}
\caption{
CPU times for {\it Example 1}
with $L=1$, $W=8$ and $1/h = 200$.
Single threaded performance.
}
\label{tab:cpu_times}
\vspace{-15pt}
\end{table}

\begin{figure}
\centering
\begin{subfigure}[t]{0.54\textwidth}
\centering
\begin{tikzpicture}
\begin{loglogaxis}[
    width=6cm,
    height=4cm,
    xmin=1e-13,
    xmax=1e4,
    x dir=reverse,
    ymin=1e-15,
    ymax=1e1,
    xlabel=target error $\algotol$,
    ylabel=$\norm{T - P_{R^{n}}T}$,
    grid=both,
    grid style={line width=.1pt, draw=gray!20},
    major grid style={line width=.2pt,draw=gray!50},
    minor xtick={1e-15, 1e-14, 1e-13, 1e-12, 1e-11, 1e-10, 1e-9, 1e-8, 1e-7, 1e-6, 1e-5, 1e-4, 1e-3, 1e-2, 1e-1, 1e0, 1e1, 1e2, 1e3},
    minor ytick={1e-15, 1e-14, 1e-13, 1e-12, 1e-11, 1e-10, 1e-9, 1e-8, 1e-7, 1e-6, 1e-5, 1e-4, 1e-3, 1e-2, 1e-1, 1e0, 1e1, 1e2, 1e3},
    xtick={1e3, 1e0, 1e-3, 1e-6, 1e-9, 1e-12},
    ytick={1e3, 1e0, 1e-3, 1e-6, 1e-9, 1e-12},
    max space between ticks=25,
    legend style={at={(1.2,1.75)},anchor=north east},
  ]
  \addplot+[densely dotted, blue,  thick, mark=none] table[x index=0,y index=5] {twosquares_adaptive_convergence.dat};
  \addplot+[densely dashed, blue,  thick, mark=none] table[x index=0,y index=4] {twosquares_adaptive_convergence.dat};
  \addplot+[solid,          black, thick, mark=none] table[x index=0,y index=3] {twosquares_adaptive_convergence.dat};
  \addplot+[densely dashed, red,   thick, mark=none] table[x index=0,y index=2] {twosquares_adaptive_convergence.dat};
  \addplot+[densely dotted, red,   thick, mark=none] table[x index=0,y index=1] {twosquares_adaptive_convergence.dat};
    \addplot+[solid, grau, thick, mark=o,mark repeat=10] table[x index=0,y index=0] {twosquares_adaptive_convergence.dat};
  \legend{max, 75 percentile, 50 percentile, 25 percentile, min,$y=x$};
\end{loglogaxis}
\end{tikzpicture}
\subcaption{\footnotesize Quartiles for 10 test vectors.}
\label{fig:adaptive performance}
\end{subfigure}
\begin{subfigure}[t]{0.44\textwidth}
\centering
\begin{tikzpicture}
\begin{loglogaxis}[
    width=6cm,
    height=4cm,
    xmin=1e-13,
    xmax=1e4,
    x dir=reverse,
    ymin=1e-15,
    ymax=1e1,
    xlabel=target error $\algotol$,
    yticklabels={,,},
    grid=both,
    grid style={line width=.1pt, draw=gray!20},
    major grid style={line width=.2pt,draw=gray!50},
    minor xtick={1e-15, 1e-14, 1e-13, 1e-12, 1e-11, 1e-10, 1e-9, 1e-8, 1e-7, 1e-6, 1e-5, 1e-4, 1e-3, 1e-2, 1e-1, 1e0, 1e1, 1e2, 1e3},
    minor ytick={1e-15, 1e-14, 1e-13, 1e-12, 1e-11, 1e-10, 1e-9, 1e-8, 1e-7, 1e-6, 1e-5, 1e-4, 1e-3, 1e-2, 1e-1, 1e0, 1e1, 1e2, 1e3},
    xtick={1e3, 1e0, 1e-3, 1e-6, 1e-9, 1e-12},
    ytick={1e3, 1e0, 1e-3, 1e-6, 1e-9, 1e-12},
    legend style={at={(1.2,1.75)},anchor=north east},
  ]
  \addplot+[mark=square,      mark repeat=10, thick] table[x index=0,y index=1] {twosquares_adaptive_num_testvecs.dat};
  \addplot+[mark=o     , mark repeat=12, thick,solid] table[x index=0,y index=2] {twosquares_adaptive_num_testvecs.dat};
  \addplot+[mark=none,   mark repeat=14, thick,solid] table[x index=0,y index=3] {twosquares_adaptive_num_testvecs.dat};
  \addplot+[mark=triangle,  mark repeat=16, thick,solid] table[x index=0,y index=4] {twosquares_adaptive_num_testvecs.dat};
  \addplot+[mark=diamond, mark repeat=18, thick,solid] table[x index=0,y index=5] {twosquares_adaptive_num_testvecs.dat};
  \addplot+[mark=x,      mark repeat=20, thick,solid] table[x index=0,y index=6] {twosquares_adaptive_num_testvecs.dat};
    \addplot+[densely dashed,grau,thick,mark=none] table[x index=0,y index=0] {twosquares_adaptive_num_testvecs.dat};
  \legend{
    $n_t = 3$,
    $n_t = 5$,
    $n_t = 10$,
    $n_t = 20$,
    $n_t = 40$,
    $n_t = 80$,
    $y=x$,
  };
\end{loglogaxis}
\end{tikzpicture}
\subcaption{\footnotesize Maximum error for given number of test vectors.}
\label{fig:number of testvecs}
\end{subfigure}
\vspace{-10pt}
\caption{\footnotesize Projection
error $\sup_{\xi \in S} \inf_{\zeta \in R^{n}} \frac{\| T\xi - \zeta \|_{R}}{\|\xi \|_{S}} = \norm{T - P_{R^{n}}T}$
over target projection error for {\it Example 1}.
Meshsize $h=1/160$.
}
\label{fig:interface_adaptive_quartiles}
\vspace{-15pt}
\end{figure}
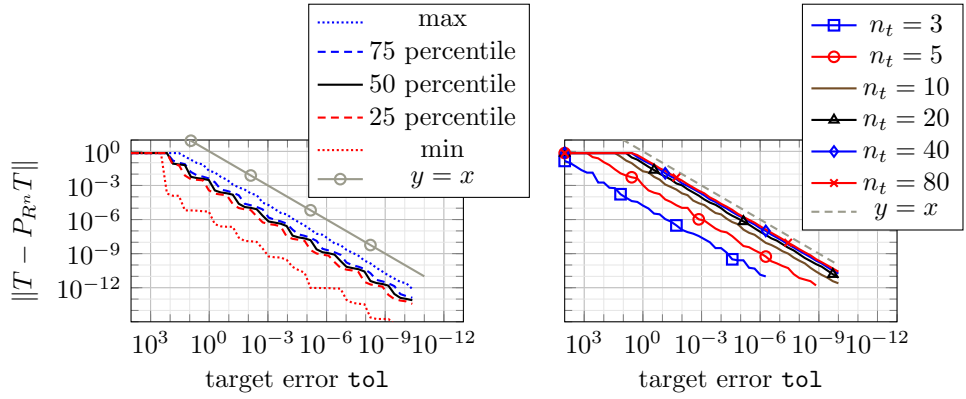

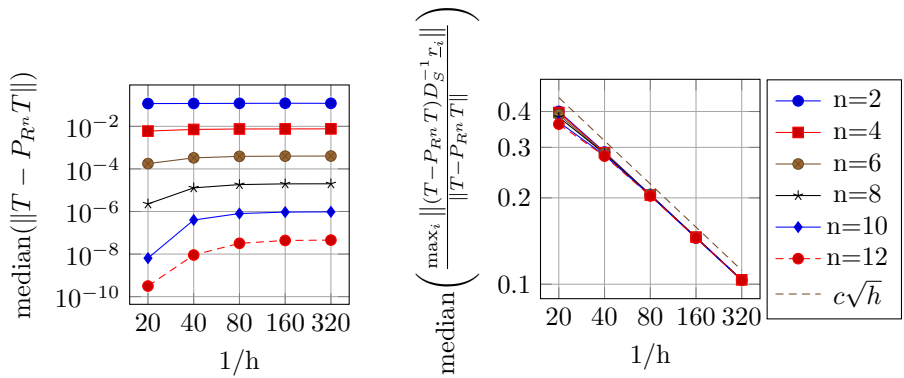
\begin{figure}
\centering
\begin{subfigure}[t]{0.4\textwidth}
\centering
\begin{tikzpicture}
\begin{loglogaxis}[
    width=4.5cm,
    height=4.5cm,
    xtick={20,40,80,160,320},
    xticklabels={20,40,80,160,320},
    ytick={1e-10,1e-8, 1e-6, 1e-4, 1e-2, 1e0},
    xlabel=1/h,
    ylabel=median($\norm{T - P_{R^n}T}$),
    legend pos=outer north east,
    grid=both,
    major grid style={line width=.2pt,draw=gray!70},
  ]
  \addplot+[] table[x index=0, y index=1] {experiment_h_deviation.dat};
  \addplot+[] table[x index=0, y index=2] {experiment_h_deviation.dat};
  \addplot+[] table[x index=0, y index=3] {experiment_h_deviation.dat};
  \addplot+[] table[x index=0, y index=4] {experiment_h_deviation.dat};
  \addplot+[] table[x index=0, y index=5] {experiment_h_deviation.dat};
  \addplot+[] table[x index=0, y index=6] {experiment_h_deviation.dat};
\end{loglogaxis}
\end{tikzpicture}
\subcaption{\footnotesize h dependency of deviation}
\label{fig:h dependency deviation}
\end{subfigure}
\begin{subfigure}[t]{0.55\textwidth}
\centering
\begin{tikzpicture}
\begin{loglogaxis}[
    width=4.5cm,
    height=4.5cm,
    xtick={20,40,80,160,320},
    xticklabels={20,40,80,160,320},
    ytick={0.1,0.2,0.3,0.4},
    yticklabels={0.1,0.2,0.3,0.4},
    xlabel=1/h,
    ylabel=median$\left(\frac{\max_i \norm{(T - P_{R^n}T)D_S^{-1} \underline r_i}}{\norm{T - P_{R^n}T}}\right)$,
    legend pos=outer north east,
    grid=both,
    major grid style={line width=.2pt,draw=gray!70},
  ]
  \addplot+[] table[x index=0, y index=1] {experiment_h_testvecs.dat};
  \addplot+[] table[x index=0, y index=2] {experiment_h_testvecs.dat};
  \addplot+[] table[x index=0, y index=3] {experiment_h_testvecs.dat};
  \addplot+[] table[x index=0, y index=4] {experiment_h_testvecs.dat};
  \addplot+[] table[x index=0, y index=5] {experiment_h_testvecs.dat};
  \addplot+[] table[x index=0, y index=6] {experiment_h_testvecs.dat};
  \addplot+[mark=none, domain=20:320] {2*x^(-0.5)};
  \legend{n=2, n=4, n=6, n=8, n=10, n=12, $c \sqrt{h}$};
\end{loglogaxis}
\end{tikzpicture}
\subcaption{\footnotesize h dependency of testvector norm}
\label{fig:h dependency testvecnorm}
\end{subfigure}
\vspace{-10pt}
\caption{\footnotesize h dependency for {\it Example 1}.
Statistics over 10,000 samples.
}
\label{fig:h_dependency}
\vspace{-15pt}
\end{figure}

\begin{figure}
\centering
\begin{subfigure}[t]{0.4\textwidth}
\centering
\begin{tikzpicture}
\begin{axis}[
    ymin=4,
    ymax=16,
    xlabel= $n_t$,
    ylabel= $\ceffshort \left/ \frac{\Delta( T - P_{R^n}T, n_t, \varepsilon_\mathrm{testfail}) }{ \norm{T - P_{R^n}T}} \right.$,
    width=5cm,
    height=5cm,
    grid style={line width=.1pt, draw=gray!20},
    major grid style={line width=.2pt,draw=gray!50},
    grid=both,
    minor xtick={0, 10, 20, 30, 40, 50, 60, 70, 80, 90, 100},
    minor ytick={4,5,6,7,8,9,10,11,12,13,14,15,16},
  ]
  \addplot+[mark=none, mark repeat=10, thick, densely dashed] table[x index=0, y expr=\thisrowno{5} / \thisrowno{3}] {twosquares_testvecnorms.dat};
\end{axis}
\end{tikzpicture}
\subcaption{\footnotesize numerical efficiency vs.~its upper bound}
\label{fig:numerical efficiency 1}
\end{subfigure}
\begin{subfigure}[t]{0.5\textwidth}
\centering
\begin{tikzpicture}
\begin{semilogyaxis}[
    xlabel= $n_t$,
    ylabel= $\frac{\max\limits_{i=1, \dots, n_t} \norm{(T - P_{R^n}T) D_S^{-1} \underline r_i}_R }{ \norm{T - P_{R^n}T}}$,
    width=5cm,
    height=5cm,
    ymin = 1e-7,
    ymax = 10,
    grid style={line width=.1pt, draw=gray!20},
    major grid style={line width=.2pt,draw=gray!50},
    grid=both,
    ytick={1e0, 1e-2, 1e-4, 1e-6},
    minor xtick={0, 10, 20, 30, 40, 50, 60, 70, 80, 90, 100},
    minor ytick={1e0, 1e-1, 1e-2, 1e-3, 1e-4, 1e-5, 1e-6},
    legend pos=outer north east,
    legend style={at={(0.98,0.02)},anchor=south east},
  ]
  \addplot+[mark=none, mark repeat=10, thick, densely dashed] table[x index=0, y index=5] {twosquares_testvecnorms.dat};
  \addplot+[mark=diamond, mark repeat=10, thick, red] table[x index=0, y index=4] {twosquares_testvecnorms.dat};
  \addplot+[mark=square, mark repeat=10, thick, black] table[x index=0, y index=3] {twosquares_testvecnorms.dat};
  \addplot+[mark=triangle, mark repeat=10, thick, gruen] table[x index=0, y index=2] {twosquares_testvecnorms.dat};
  \addplot+[mark=none, mark repeat=10, thick, densely dotted] table[x index=0, y index=1] {twosquares_testvecnorms.dat};
  \legend{upper limit, max, median, min, lower limit};
\end{semilogyaxis}
\end{tikzpicture}
\subcaption{\footnotesize normalized test vector norm}
\label{fig:numerical efficiency 2}
\end{subfigure}
\vspace{-10pt}
\caption{\footnotesize
Numerical efficiency for {\it Example 1}.
}
\label{fig:nt_dependency}
\vspace{-15pt}
\end{figure}
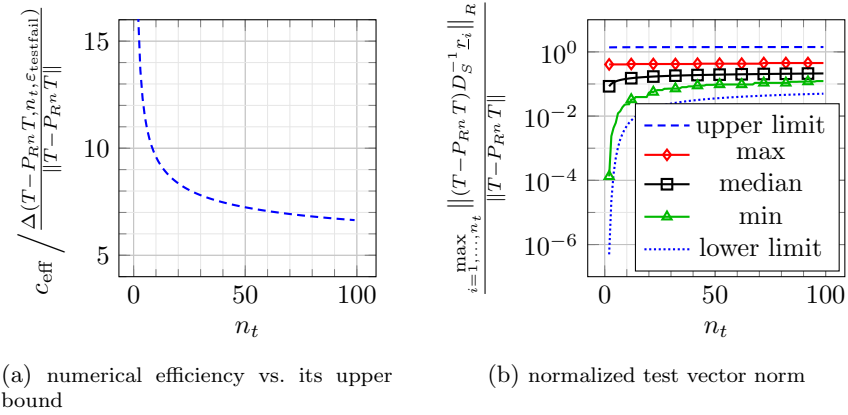

\subsection{Helmholtz equation}\label{subsect:helmholtz}
In this subsection we analyze the behavior of the proposed algorithm
in a numerical test case approximating the solution of
the Helmholtz equation. The domain $\Omega$, the boundaries $\Gamma_{in}$
and $\Gamma_{out}$ and the boundary conditions are the same as in \cref{subsect:analytic},
only the operator $\mathcal{A}$ differs and is defined as
$\mathcal{A} = -\Delta - \wavenumber^2$ in this subsection.
As for {\it Example 1}, it has 51,681 DOFs, $N_S = 322$ and $N_R= 161$ for $1/h=160$.
We refer to this numerical example as {\it Example 2}.
We assume the problem to be inf-sup stable and thus uniquely solvable,
which is the case as long as it is not in a resonant configuration.
A treatment of the resonant case is beyond the scope of this publication.

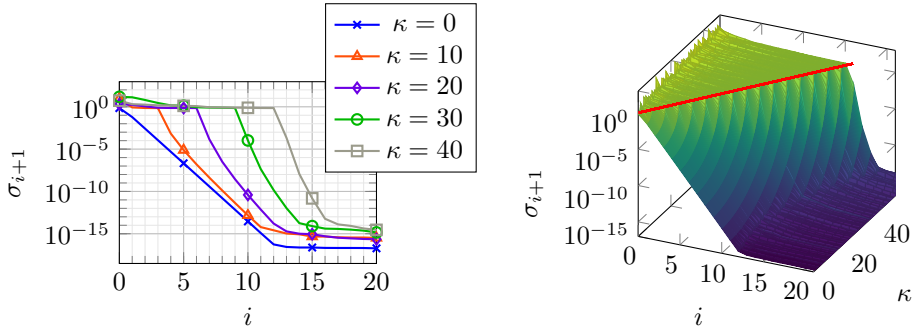
\begin{figure}
\centering
\begin{subfigure}[t]{0.5\textwidth}
\begin{tikzpicture}
\begin{semilogyaxis}[%
    width=5cm,
    height=4cm,
    legend pos=outer north east,
    legend style={at={(0.8,0.5)},anchor=south west},
    xmin=0,
    xmax=20,
    xtick={0, 5, 10, 15, 20},
    ytick={1e-15, 1e-10, 1e-5, 1e0},
    xlabel=$i$,
    ylabel=$\sigma_{i+1}$,
    grid=both,
    grid style={line width=.1pt, draw=gray!20},
    major grid style={line width=.2pt,draw=gray!50},
    minor xtick={1,2,3,4,5,6,7,8,9,10,11,12,13,14,15,16,17,18,19,20},
    minor ytick={1e-15, 1e-14, 1e-13, 1e-12, 1e-11, 1e-10, 1e-9, 1e-8, 1e-7, 1e-6, 1e-5, 1e-4, 1e-3, 1e-2, 1e-1, 1e0, 1e1, 1e2, 1e3},
]
  \addplot+[blue, mark=x, mark repeat=5, thick] table[x index=0, y index=1] {helmholtz_selected.dat};
  \addplot+[oran, mark=triangle, mark repeat = 5, thick] table[x index=0, y index=2] {helmholtz_selected.dat};
  \addplot+[lila, mark=diamond, mark repeat = 5, thick] table[x index=0, y index=3] {helmholtz_selected.dat};
  \addplot+[gruen, mark=o, mark repeat=5, thick] table[x index=0, y index=4] {helmholtz_selected.dat};
  \addplot+[grau, mark=square, mark repeat=5, thick] table[x index=0, y index=5] {helmholtz_selected.dat};
  \legend{%
    $\wavenumber=0$,
    $\wavenumber=10$,
    $\wavenumber=20$,
    $\wavenumber=30$,
    $\wavenumber=40$,
  }
\end{semilogyaxis}
\end{tikzpicture}
\end{subfigure}
\begin{subfigure}[t]{0.45\textwidth}
\centering
\includegraphics{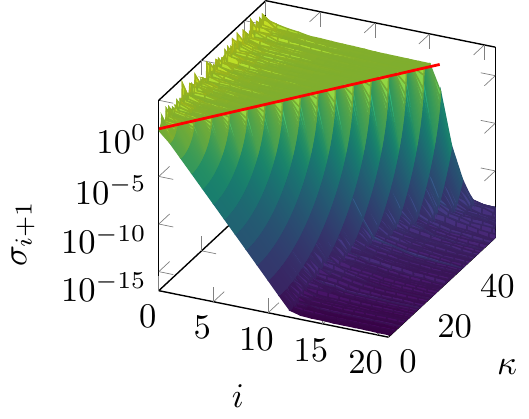}
\end{subfigure}
\caption{\footnotesize
Singular value decay for {\it Example 2}.
The red line at $i=\wavenumber/\pi$ in the right plot marks the observed length of the plateau.
}
\label{fig:helmholtz_svd}
\vspace{-15pt}
\end{figure}

\begin{figure}
\centering
\begin{subfigure}[t]{0.55\textwidth}
\centering
\begin{tikzpicture}
\begin{semilogyaxis}[
    width=5cm,
    height=5cm,
    xmin=-1,
    xmax=21,
    ymin=1e-17,
    ymax=1e4,
    xlabel=basis size $n$,
    ylabel=$\norm{T - P_{R^{n}}T}$,
    grid=both,
    grid style={line width=.1pt, draw=gray!20},
    major grid style={line width=.2pt,draw=gray!50},
    minor xtick={1,2,3,4,5,6,7,8,9,10,11,12,13,14,15,16,17,18,19,20},
    minor ytick={1e-15, 1e-14, 1e-13, 1e-12, 1e-11, 1e-10, 1e-9, 1e-8, 1e-7, 1e-6, 1e-5, 1e-4, 1e-3, 1e-2, 1e-1, 1e0, 1e1, 1e2, 1e3},
    ytick={1e0, 1e-5, 1e-10, 1e-15},
    max space between ticks=25,
    legend style={at={(1.45,1.45)},anchor=north east},
  ]
  \addplot+[densely dotted, blue,  mark=o, mark options={solid}] table[x expr=\coordindex,y index=5] {helmholtz_percentiles.dat};
  \addplot+[densely dashed, blue,  mark=diamond, mark options={solid}] table[x expr=\coordindex,y index=4] {helmholtz_percentiles.dat};
  \addplot+[densely dashed, blue,  mark=x, mark options={solid}] table[x expr=\coordindex,y index=3] {helmholtz_percentiles.dat};
  \addplot+[solid,          black, thick, mark=x] table[x expr=\coordindex,y index=2] {helmholtz_percentiles.dat};
  \addplot+[densely dashed, gruen, mark=x, mark options={solid}] table[x expr=\coordindex,y index=1] {helmholtz_percentiles.dat};
  \addplot+[densely dotted, gruen, mark=o, mark options={solid}] table[x expr=\coordindex,y index=0] {helmholtz_percentiles.dat};
  \addplot+[solid, red, mark=0, thick] table[x index=0, y index=4] {helmholtz_selected.dat};
  \legend{max, 99 percentile, 75 percentile, 50 percentile, 25 percentile, min, $\sigma_{i+1}$};
\end{semilogyaxis}
\end{tikzpicture}
\subcaption{\footnotesize Percentiles, worst case, and best case}
\label{fig:helmholtz deviation percentiles}
\end{subfigure}
\begin{subfigure}[t]{0.4\textwidth}
\centering
\begin{tikzpicture}
\begin{semilogyaxis}[
    width=5cm,
    height=5cm,
    xmin=-1,
    xmax=21,
    ymin=1e-17,
    ymax=1e4,
    xlabel=basis size $n$,
    yticklabels={,,},
    grid=both,
    grid style={line width=.1pt, draw=gray!20},
    major grid style={line width=.2pt,draw=gray!50},
    minor xtick={1,2,3,4,5,6,7,8,9,10,11,12,13,14,15,16,17,18,19,20},
    minor ytick={1e-16, 1e-15, 1e-14, 1e-13, 1e-12, 1e-11, 1e-10, 1e-9, 1e-8, 1e-7, 1e-6, 1e-5, 1e-4, 1e-3, 1e-2, 1e-1, 1e0, 1e1, 1e2, 1e3},
    ytick={1e0, 1e-5, 1e-10, 1e-15},
    max space between ticks=25,
    legend style={at={(1.1,1.35)},anchor=north east},
  ]
  \addplot+[] table[x index=0, y index=1]{helmholtz_expectation.dat};
  \addplot+[solid,          black, thick, mark=x] table[x expr=\coordindex,y index=0] {helmholtz_means.dat};
  \legend{a priori limit for mean, mean};
\end{semilogyaxis}
\end{tikzpicture}
\subcaption{\footnotesize Mean of deviation and a priori limit}
\label{fig:helmholtz deviation mean}
\end{subfigure}
\vspace{-10pt}
\caption{\footnotesize Projection 
error $\sup_{\xi \in S} \inf_{\zeta \in R^{n}} \frac{\| T\xi - \zeta \|_{R}}{\|\xi \|_{S}} = \norm{T - P_{R^{n}}T}$
over basis size $n$ for {\it Example 2} with $\wavenumber=30$.
Meshsize $h=1/160$.\\
}
\label{fig:helmholtz_a_priori}
\vspace{-15pt}
\end{figure}

\begin{figure}
\centering
\begin{subfigure}[t]{0.54\textwidth}
\centering
\begin{tikzpicture}
\begin{loglogaxis}[
    width=6cm,
    height=4cm,
    xmin=1e-13,
    xmax=1e4,
    x dir=reverse,
    ymin=1e-15,
    ymax=1e1,
    xlabel=target error $\algotol$,
    ylabel=$\norm{T - P_{R^{n}}T}$,
    grid=both,
    grid style={line width=.1pt, draw=gray!20},
    major grid style={line width=.2pt,draw=gray!50},
    minor xtick={1e-15, 1e-14, 1e-13, 1e-12, 1e-11, 1e-10, 1e-9, 1e-8, 1e-7, 1e-6, 1e-5, 1e-4, 1e-3, 1e-2, 1e-1, 1e0, 1e1, 1e2, 1e3},
    minor ytick={1e-15, 1e-14, 1e-13, 1e-12, 1e-11, 1e-10, 1e-9, 1e-8, 1e-7, 1e-6, 1e-5, 1e-4, 1e-3, 1e-2, 1e-1, 1e0, 1e1, 1e2, 1e3},
    xtick={1e3, 1e0, 1e-3, 1e-6, 1e-9, 1e-12},
    ytick={1e3, 1e0, 1e-3, 1e-6, 1e-9, 1e-12},
    max space between ticks=25,
    legend style={at={(1.2,1.75)},anchor=north east},
  ]
  \addplot+[densely dotted, blue,  thick, mark=none] table[x index=0,y index=5] {helmholtz_adaptive_convergence.dat};
  \addplot+[densely dashed, blue,  thick, mark=none] table[x index=0,y index=4] {helmholtz_adaptive_convergence.dat};
  \addplot+[solid,          black, thick, mark=none] table[x index=0,y index=3] {helmholtz_adaptive_convergence.dat};
  \addplot+[densely dashed, red,   thick, mark=none] table[x index=0,y index=2] {helmholtz_adaptive_convergence.dat};
  \addplot+[densely dotted, red,   thick, mark=none] table[x index=0,y index=1] {helmholtz_adaptive_convergence.dat};
    \addplot+[solid, grau, thick, mark=o,mark repeat=10] table[x index=0,y index=0] {helmholtz_adaptive_convergence.dat};
  \legend{max, 75 percentile, 50 percentile, 25 percentile, min,$y=x$};
\end{loglogaxis}
\end{tikzpicture}
\subcaption{\footnotesize Quartiles for 10 test vectors.}
\label{fig:helmholtz adaptive performance}
\end{subfigure}
\begin{subfigure}[t]{0.44\textwidth}
\centering
\begin{tikzpicture}
\begin{loglogaxis}[
    width=6cm,
    height=4cm,
    xmin=1e-13,
    xmax=1e4,
    x dir=reverse,
    ymin=1e-15,
    ymax=1e1,
    xlabel=target error $\algotol$,
    yticklabels={,,},
    grid=both,
    grid style={line width=.1pt, draw=gray!20},
    major grid style={line width=.2pt,draw=gray!50},
    minor xtick={1e-15, 1e-14, 1e-13, 1e-12, 1e-11, 1e-10, 1e-9, 1e-8, 1e-7, 1e-6, 1e-5, 1e-4, 1e-3, 1e-2, 1e-1, 1e0, 1e1, 1e2, 1e3},
    minor ytick={1e-15, 1e-14, 1e-13, 1e-12, 1e-11, 1e-10, 1e-9, 1e-8, 1e-7, 1e-6, 1e-5, 1e-4, 1e-3, 1e-2, 1e-1, 1e0, 1e1, 1e2, 1e3},
    xtick={1e3, 1e0, 1e-3, 1e-6, 1e-9, 1e-12},
    ytick={1e3, 1e0, 1e-3, 1e-6, 1e-9, 1e-12},
    legend style={at={(1.2,1.75)},anchor=north east},
  ]
  \addplot+[mark=square,      mark repeat=10, thick] table[x index=0,y index=1] {helmholtz_adaptive_num_testvecs.dat};
  \addplot+[mark=o     , mark repeat=12, thick,solid] table[x index=0,y index=2] {helmholtz_adaptive_num_testvecs.dat};
  \addplot+[mark=none,   mark repeat=14, thick,solid] table[x index=0,y index=3] {helmholtz_adaptive_num_testvecs.dat};
  \addplot+[mark=triangle,  mark repeat=16, thick,solid] table[x index=0,y index=4] {helmholtz_adaptive_num_testvecs.dat};
  \addplot+[mark=diamond, mark repeat=18, thick,solid] table[x index=0,y index=5] {helmholtz_adaptive_num_testvecs.dat};
  \addplot+[mark=x,      mark repeat=20, thick,solid] table[x index=0,y index=6] {helmholtz_adaptive_num_testvecs.dat};
    \addplot+[densely dashed,grau,thick,mark=none] table[x index=0,y index=0] {helmholtz_adaptive_num_testvecs.dat};
  \legend{
    $n_t = 3$,
    $n_t = 5$,
    $n_t = 10$,
    $n_t = 20$,
    $n_t = 40$,
    $n_t = 80$,
    $y=x$,
  };
\end{loglogaxis}
\end{tikzpicture}
\subcaption{\footnotesize Maximum error for given number of test vectors.}
\label{fig:helmholtz number of testvecs}
\end{subfigure}
\vspace{-10pt}
\caption{\footnotesize Projection
error $\sup_{\xi \in S} \inf_{\zeta \in R^{n}} \frac{\| T\xi - \zeta \|_{R}}{\|\xi \|_{S}} = \norm{T - P_{R^{n}}T}$
over target projection error for {\it Example 2} with $\wavenumber=30$.
Meshsize $h=1/160$.
}
\label{fig:helmholtz_interface_adaptive_quartiles}
\vspace{-15pt}
\end{figure}

For $\wavenumber=0$ we obtain {\it Example 1}.
We observe that the singular values of the transfer operator
first have a plateau and then decay exponentially,
see \cref{fig:helmholtz_svd}.
The longer the plateau, the faster is the exponential decay.
The length of the plateau is observed to be very close
to the length of the inner interface divided by 
a half wavelength, i.e.~ $1/(\lambda / 2) = \wavenumber / \pi$.
Comparing this with the
analysis of Finite Element methods for
the Helmholtz equation
(cf.~\cite{ihlenburg2006finite}),
one finds this similar to the
``minimal resolution condition''
$1/h \geq \sqrt{12}/\wavenumber$.

Algorithm \ref{algo:adaptive_range_approximation}
succeeds to generate reduced spaces $R^n$
which achieve a projection error $\norm{T-P_{R^n}}$
which is close the the optimal projection error
given by the singular values of the transfer operator.
We show results for $\wavenumber=30$ in \cref{fig:helmholtz_a_priori}.
Also in the adaptive case, we observe the expected behavior, see \cref{fig:helmholtz adaptive performance}
and \cref{fig:helmholtz number of testvecs}.
The plateaus which can be observed in \cref{fig:helmholtz adaptive performance}
are due to the very fast decay of the singular values.
E.g.~the first plateau is at an error of about $10^{-3}$, which
is the error usually achieved at a basis size of 10 (cf.~\cref{fig:helmholtz deviation percentiles}).
The next plateau at an error of about $10^{-7}$ corresponds to a basis size of 11.
\subsection{A transfer operator with slowly decaying singular values; application to linear elasticity}\label{subsect:linearelasticity}
In this subsection we numerically analyze Algorithm \ref{algo:adaptive_range_approximation} and the theoretical findings of \cref{sect:randomized_la} for a numerical test case {\it Example 3} where the singular values of the transfer operator exhibit a relatively slow decay and $N$, $N_{S}$, and $N_{R}$ are relatively large. Moreover, we shortly illustrate that Algorithm \ref{algo:adaptive_range_approximation} is attractive from a computational viewpoint and with respect to memory requirement. 

To that end let $ \Omega_{in} = (-0.5,0.5) \times (-0.5,0.5) \times (-0.5,0.5)$ be the subdomain on which we aim to construct a local approximation space, $\Omega = (-2,2) \times (-0.5,0.5) \times (-2,2)$ the (oversampling) domain, and $\Gamma_{out}=\{-2,2\} \times (-0.5,0.5) \times (-2,2) \cup (-2,2) \times (-0.5,0.5) \times \{-2,2\}$ the outer boundary. On $\partial\Omega \setminus \Gamma_{out}$ we prescribe homogeneous Neumann boundary conditions and we suppose that $\Omega$ does not border the Dirichlet boundary of $\Omega_{gl}$. We assume that $\Omega$ represents an isotropic homogeneous material and we consider the equations of linear elasticity. Therefore, we choose $\mathcal{X}=[H^{1}(\Omega)]^{3}$, $\mathcal{X}_{0}=\{ v \in [H^{1}(\Omega)]^{3}\,:\, v|_{\Gamma_{out}} = 0\}$, $
\mathcal{A}: \mathcal{X} \rightarrow \mathcal{X}_{0}'$, $\mathcal{A}u = - \nabla C : \varepsilon(u)$ for $u \in \mathcal{X}$ and consider the following boundary value problem: Find $u \in \mathcal{X}$ such that
\begin{equation}\label{eq:PDE linear elasticity}
\mathcal{A}u = 0 \quad \text{in} \enspace \mathcal{X}_{0}'
\end{equation}
with arbitrary Dirichlet boundary conditions on $\Gamma_{out}$. Here, we set Young's modulus equal to one, $C$ is the fourth-order stiffness tensor
\begin{equation*}
C_{ijkl} = \frac{\nu}{(1+\nu)(1 -2\nu)}\delta_{ij}\delta_{kl} + \frac{1}{2(1 + \nu)}(\delta_{ik}\delta_{jl} + \delta_{il}\delta_{jk}), \quad 1 \leq i,j,k,l \leq 3,
\end{equation*}
where $\delta_{ij}$ denotes the Kronecker delta, and we choose Poisson's ratio $\nu =0.3$.
Moreover, $\varepsilon(u) = 0.5 (\nabla u + (\nabla u)^{T})$ is the infinitesimal strain tensor and the colon operator $:$ is defined as $C : \varepsilon(u) = \sum_{k,l=1}^{3}C_{ijkl}\varepsilon_{kl}(u)$. 
\begin{figure}
\centering
\begin{subfigure}[t]{0.4\textwidth}
\centering
\begin{tikzpicture}
\begin{semilogyaxis}[
    width=4.5cm,
    height=4cm,
    xmin=2,
    xmax=300,
    ymin=1e-9,
    ymax=150,
    xlabel=$k$,
    ylabel=$\mathbb{E}(\|T-P_{R^{k+p}}T\|)$,
    grid=both,
    grid style={line width=.1pt, draw=gray!20},
    major grid style={line width=.2pt,draw=gray!50},
    minor xtick={2,25,50,75,100,125,150,175,200,225,250,275,300,325,350},
    minor ytick={1e-9, 1e-8, 1e-7, 1e-6, 1e-5, 1e-4, 1e-3, 1e-2, 1e-1, 1e0, 1e1, 1e2},
    xtick={2,100,200,300},
    ytick={1e2, 1e0, 1e-2, 1e-4, 1e-6, 1e-8},
    max space between ticks=25,
    legend style={at={(1.1,1.55)},anchor=north east},
  ]
  \addplot+[solid, black, thick, mark=o, mark repeat=50] table[x index=0,y index=2] {elasticity_cube_plate_lambda.dat};
    \addplot+[solid,          gruen, thick, mark=square,mark repeat=50] table[x index=0,y index=1] {elasticity_cube_plate_a_priori_bound.dat};
  \addplot+[solid, blue, thick, mark=none] table[x index=0,y index=2] {elasticity_cube_plate_a_priori_bound.dat};
  \addplot+[densely dotted, red, thick, mark=none] table[x index=0,y index=1] {elasticity_cube_plate_lambda.dat};
  \addplot+[densely dashed, gelb, thick, mark=none] table[x index=0,y index=1] {elasticity_cube_plate_error.dat};
\end{semilogyaxis}
\end{tikzpicture}
\subcaption{\footnotesize Convergence behavior on $\Omega$}
\label{fig: linear elasticity a priori 1}
\end{subfigure}
\begin{subfigure}[t]{0.55\textwidth}
\centering
\begin{tikzpicture}
\begin{semilogyaxis}[
    width=4.5cm,
    height=4cm,
    xmin=2,
    xmax=200,
    ymin=1e-9,
    ymax=150,
    xlabel=$k$,
    yticklabels={,,},
    grid=both,
    grid style={line width=.1pt, draw=gray!20},
    major grid style={line width=.2pt,draw=gray!50},
    minor xtick={2,25,50,75,100,125,150,175,200},
    minor ytick={1e-9, 1e-8, 1e-7, 1e-6, 1e-5, 1e-4, 1e-3, 1e-2, 1e-1, 1e0, 1e1, 1e2},
    xtick={2,50,100,150,200},
    ytick={1e2, 1e0, 1e-2, 1e-4, 1e-6, 1e-8},
    max space between ticks=25,
    legend pos=outer north east,
  ]
 \addplot+[solid, black, thick, mark=o,mark repeat=50] table[x index=0,y index=2] {elasticity_plate_lambda.dat};
   \addplot+[solid,          gruen, thick, mark=square, mark repeat=50] table[x index=0,y index=1] {elasticity_plate_a_priori_bound.dat};
  \addplot+[solid, blue, thick, mark=none] table[x index=0,y index=2] {elasticity_plate_a_priori_bound.dat};
  \addplot+[densely dotted, red, thick, mark=none] table[x index=0,y index=1] {elasticity_plate_lambda.dat};
  \addplot+[densely dashed, gelb, thick, mark=none] table[x index=0,y index=1] {elasticity_plate_error.dat};
  \legend{$\sigma_{k+1}$, a priori, sc. a priori, $\sqrt{k}\sigma_{k+1}$,  $\mathbb{E}(\|T-P_{R^{k+p}}T\|)$};
\end{semilogyaxis}
\end{tikzpicture}
\subcaption{\footnotesize Convergence behavior on $\widehat{\Omega}$.}
\label{fig: linear elasticity a priori 2}
\end{subfigure}
\vspace{-10pt}
\caption{\footnotesize Comparison of the convergence behavior of $\sigma_{k+1}$, $\sqrt{k}\sigma_{k+1}$, $\mathbb{E}(\|T-P_{R^{k+p}}T\|)$, the a priori error bound as introduced in \eqref{eq:a priori mean}, and the a priori error bound of \eqref{eq:a priori mean} scaled with a constant such that its value for $k=2$ equals the one of $\mathbb{E}(\|T-P_{R^{k+p}}T\|)$ (sc. a priori) for increasing $k$ for and $p=2$ for the oversampling domains $\Omega$ (a) and $\widehat{\Omega}$ (b).}
\label{fig:linear elasticity a priori}
\vspace{-15pt}
\end{figure}
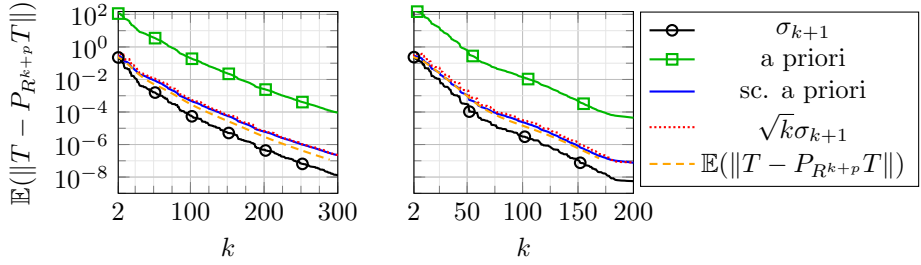

For the FE discretization we use a regular mesh with hexahedral elements and a mesh size $h=0.1$ in each space direction and a corresponding FE space $X$ with linear FE resulting in 
$\dim(X)=N=55473$, $\dim(R)=N_{R}=3987$, and $\dim(S)=N_{S}=5280$. Note that although in theory we should subtract the orthogonal projection on the six rigid body motions from the FE basis functions, in actual practice we avoid that by subtracting the orthogonal projection from the $A$-harmonic extensions only. Finally, we equip the source space $S$ with the $L^{2}$-inner product and the range space $R$ with the energy inner product
\begin{equation*}
(w,v)_{R} := \int_{\Omega_{in}} \frac{\partial w^{i}}{\partial x_{j}} C_{ijkl} \frac{\partial v^{k}}{\partial x_{l}} dx.  
\end{equation*}

Analyzing the convergence behavior of $\mathbb{E}(\|T - P_{R^{k+p}}T\|)$ for a growing number of randomly generated basis functions $k$ and a (fixed) oversampling parameter $p=2$ in \cref{fig: linear elasticity a priori 1} we observe that the local approximation spaces generated as proposed in \cref{sect:randomized_la} yield an approximation that converges nearly with the optimal rate $\sigma_{k+1}$. Moreover, we see in \cref{fig: linear elasticity a priori 1} that the a priori error bound as proposed in \cref{eq:a priori mean} reproduces the convergence behavior of $\mathbb{E}(\|T - P_{R^{k+p}}T\|)$ quite well, where the mean of the deviation converges slightly faster than the a priori error bound. Furthermore, for the present test case the a priori error bound seems to behave like $\sqrt{k}\sigma_{k+1}$, arguing that the latter might be the dominating factor. Finally, we observe that the a priori bound is rather pessimistic as it overestimates $\mathbb{E}(\|T - P_{R^{k+p}}T\|)$ by a factor of more than $100$. This is mainly due to the square root of the conditions of the inner product matrices which amount to $(\lambda_{max}^{\underline{M}_{R}}/\lambda_{min}^{\underline{M}_{R}})^{1/2}\approx 17.3197$ and $(\lambda_{max}^{\underline{M}_{S}}/\lambda_{min}^{\underline{M}_{S}})^{1/2}\approx 3.4404$. 

If we consider a flatter domain $\widehat{\Omega}=(-2,2)\times (-0.25,0.25) \times (-2,2)$ and a flatter subdomain $\widehat{\Omega}_{in}=(-0.5,0.5)\times (-0.25,0.25) \times (-0.5,0.5)$ instead, where $\widehat{\Gamma}_{out}=\{-2,2\} \times (-0.25,0.25) \times (-2,2) \cup (-2,2) \times (-0.25,0.25) \times \{-2,2\}$ and we still consider the same PDE and the same inner products as above, we observe in \cref{fig: linear elasticity a priori 2} that until $k\approx 75$ the a priori bound reproduces the convergence behavior of $\mathbb{E}(\|T - P_{R^{k+p}}T\|)$ perfectly.
We may thus conclude that the a priori bound in \cref{eq:a priori mean} seems to be sharp regarding the convergence behavior of $\mathbb{E}(\|T-P_{R^{k+p}}T\|)$ in the basis size $k$.
The a priori estimates could be improved slightly by finding the optimal oversampling size $p$, which was
fixed to its mimimum value of $2$ in this experiment. Expecially on the domain $\Omega$, where the singular
values of the transfer operator show a slower decay, a larger oversampling would be beneficial.
For the computations on $\widehat{\Omega}$ we employed again a regular hexaedral mesh with $h=0.1$ and linear FE with $N=30258$, $N_{R}=2172$, and $N_{S}=2880$. Finally, for all results in this subsection we computed the statistics over $1000$ samples. From now on all results are computed on $\Omega$.
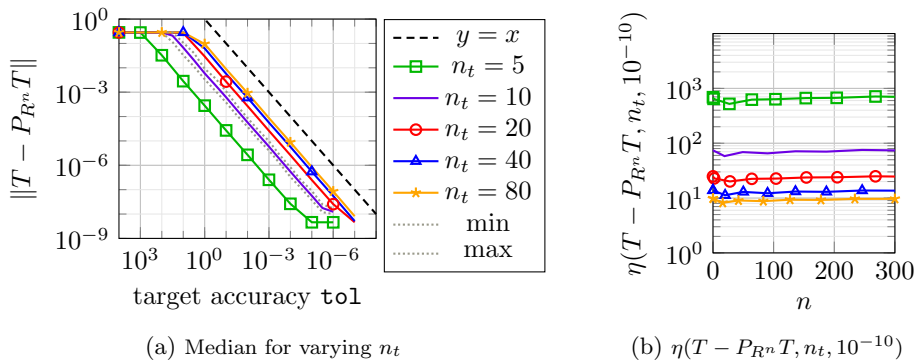
\begin{figure}
\centering
\begin{subfigure}[t]{0.59\textwidth}
\centering
\begin{tikzpicture}
\begin{loglogaxis}[
    width=5cm,
    height=4.5cm,
    xmin=1e-8,
    xmax=1e4,
    x dir=reverse,
    ymin=1e-9,
    ymax=1e0,
    xlabel=target accuracy \texttt{tol},
    ylabel=$\|T-P_{R^{n}}T\|$,
    grid=both,
    grid style={line width=.1pt, draw=gray!20},
    major grid style={line width=.2pt,draw=gray!50},
    minor xtick={1e-15, 1e-14, 1e-13, 1e-12, 1e-11, 1e-10, 1e-9, 1e-8, 1e-7, 1e-6, 1e-5, 1e-4, 1e-3, 1e-2, 1e-1, 1e0, 1e1, 1e2, 1e3},
    minor ytick={1e-15, 1e-14, 1e-13, 1e-12, 1e-11, 1e-10, 1e-9, 1e-8, 1e-7, 1e-6, 1e-5, 1e-4, 1e-3, 1e-2, 1e-1, 1e0, 1e1, 1e2, 1e3},
    xtick={1e3,1e0, 1e-3, 1e-6},
    ytick={1e3, 1e0, 1e-3, 1e-6, 1e-9, 1e-12},
    legend pos=outer north east,
  ]
  \addplot+[densely dashed,black,thick,mark=none] table[x index=0,y index=0] {elasticity_test_vectors_y_x.dat};
  \addplot+[solid,gruen,thick,mark=square] table[x index=0,y index=1] {elasticity_test_vectors_r_fuenf.dat};
  \addplot+[solid,lila,thick,mark=none] table[x index=0,y index=1] {elasticity_test_vectors_r_zehn.dat};
 \addplot+[solid,red,thick,mark=o, mark repeat=5] table[x index=0,y index=1] {elasticity_test_vectors_r_rest.dat};
  \addplot+[solid,blue,thick,mark=triangle, mark repeat=3] table[x index=0,y index=2] {elasticity_test_vectors_r_rest.dat};
  \addplot+[solid,gelb,thick,mark=star, mark repeat=2] table[x index=0,y index=3] {elasticity_test_vectors_r_rest.dat};
    \addplot+[densely dotted, grau, thick, mark=none] table[x index=0,y index=1] {elasticity_adaptive_prctile.dat};
  \addplot+[densely dotted, grau, thick, mark=none] table[x index=0,y index=5] {elasticity_adaptive_prctile.dat};
  \legend{
    $y=x$,
    $n_t=5$,
    $n_t=10$,
    $n_t=20$,
    $n_t=40$,
    $n_t=80$,
    min,
    max,
  };
\end{loglogaxis}
\end{tikzpicture}
\subcaption{\footnotesize Median for varying $n_{t}$}
\label{fig: linear elasticity adaptive algorithm 2}
\end{subfigure}
\begin{subfigure}[t]{0.4\textwidth}
\centering
\begin{tikzpicture}
\begin{semilogyaxis}[
    width=4cm,
    height=4.5cm,
    xmin=0,
    xmax=300,
    ymin=1,
    ymax=10000,
    xlabel=$n$,
    ylabel={$\eta(T-P_{R^{n}}T,n_{t},10^{-10})$},
    grid=both,
    grid style={line width=.1pt, draw=gray!20},
    major grid style={line width=.2pt,draw=gray!50},
    xtick={0,100,200,300},
    ytick={1,10,100,1000},
 legend style={at={(1.1,1.55)},anchor=north east},
  ]
  \addplot+[solid,gruen,thick,mark=square] table[x index=0,y index=1] {elasticity_effectivity_varying_n_r_fuenf.dat};
  \addplot+[solid,lila,thick,mark=none] table[x index=0,y index=1] {elasticity_effectivity_varying_n_r_zehn.dat};
  \addplot+[solid,red,thick,mark=o] table[x index=0,y index=1] {elasticity_effectivity_varying_n_r_zwanzig.dat};
  \addplot+[solid,blue,thick,mark=triangle] table[x index=0,y index=1] {elasticity_effectivity_varying_n_r_vierzig.dat};
  \addplot+[solid,gelb,thick,mark=star] table[x index=0,y index=1] {elasticity_effectivity_varying_n_r_achtzig.dat};
\end{semilogyaxis}
\end{tikzpicture}
\subcaption{\footnotesize $\eta(T-P_{R^{n}}T,n_{t},10^{-10})$}
\label{fig: linear elasticity a posteriori error estimator 1}
\end{subfigure}
\vspace{-10pt}
\caption{\footnotesize %
Median of the projection error $\|T-P_{R^{n}}T\|$ for a decreasing target accuracy \texttt{tol} for a varying number of test vectors $n_{t}$ and the minimal and maximal values for $n_t=10$ (a). Behavior of the median of the effectivity $\eta(T-P_{R^{n}}T,n_{t},10^{-10})$ as defined in \cref{eq:effectivity} for growing $n$ and various number of test vectors $n_{t}$ (b).}
\vspace{-15pt}
\end{figure}

Regarding the performance of Algorithm \ref{algo:adaptive_range_approximation} we first observe in \cref{fig: linear elasticity adaptive algorithm 2} that the actual error $\|T-P_{R^{n}}T\|$ lies below the target tolerance \texttt{tol} for all $1000$ samples for $n_t=10$; which holds also true for all other considered values of $n_t$. Here, we prescribe $\varepsilon_\mathrm{algofail} = 10^{-10}$ and use $3993$ as an upper bound for $N_{T}$ throughout this subsection. %
Compared with the performance of Algorithm \ref{algo:adaptive_range_approximation} for {\it Example 1} in \cref{fig:adaptive performance} the dispersion in \cref{fig: linear elasticity adaptive algorithm 2} is much smaller. This may be explained by the much faster decay of the singular values of $T$ and therefore $\|T-P_{R^{n}}T\|$ in \cref{subsect:analytic} compared with the present test case.

Similarly to \cref{fig:number of testvecs} and \cref{fig:helmholtz number of testvecs} in \cref{subsect:analytic} and \cref{subsect:helmholtz} we see in \cref{fig: linear elasticity adaptive algorithm 2} that increasing the number of test vectors $n_{t}$ from $5$ to $10$ or from $10$ to $20$ increases the ratio between the median of the actual error $\|T-P_{R^{n}}T\|$ and the target accuracy $\algotol$ significantly --- for the former by more than one magnitude --- while an increase from $n_{t}=40$ to $n_{t}=80$ has hardly any influence. This can be explained by the scaling of the effectivity of the employed a posteriori error estimator defined in \cref{eq:effectivity} which we will elaborate on shortly. Regarding the choice of $n_t$ it seems that for the present test case a value of about $20$ is in the sweet spot.
We summarize and emphasize that also in the present test case, where we have a rather slow convergence of the singular values of $T$ and thus the error $\|T-P_{R^{n}}T\|$, we need only very few local solutions in addition to the optimal amount required, demonstrating that Algorithm \ref{algo:adaptive_range_approximation} performs nearly optimally in terms of computational complexity for the current problem. This is due both to the nearly optimal convergence behavior as discussed above and the good effectivity of the a posteriori error estimator $\Delta(T-P_{R^{n}}T,n_{t},\varepsilon_{\mathrm{testfail}})$ also for few numbers of test vectors $n_{t}$, which will be addressed next.

Analyzing the effectivity $\eta(T-P_{R^{n}}T,n_{t},10^{-10})$ as defined in \cref{eq:effectivity} for  $\varepsilon_{\mathrm{testfail}}=10^{-10}$ and growing $n$ we see in \cref{fig: linear elasticity a posteriori error estimator 1} that for $n_{t} \geq 20$ the effectivity is in the order of $10$ and the a posteriori error estimator $\Delta(T-P_{R^{n}}T,n_{t},\varepsilon_{\mathrm{testfail}})$ therefore provides a sharp bound also for this test case. Moreover, the decrease of the effectivity for growing $n_{t}$ as can be observed in \cref{fig: linear elasticity a posteriori error estimator 1} explains the increase of the ratio between the median of the actual error $\|T-P_{R^{n}}T\|$ and the target accuracy $\algotol$ in \cref{fig: linear elasticity adaptive algorithm 2}. Finally, the effectivity varies only very slightly if $n$ changes and we may thus confirm that, as expected, the effectivity does not seem to depend on the basis size $n$. 

\subsection{Building a global approximation with the GFEM}\label{subsect:GFEM}
In order to successfully apply the proposed algorithm in the context of a method, 
it has to be possible to define a transfer operator with quickly
decaying spectrum.
Moreover, allowing to bound the local approximation error  
in terms of $\norm{T - P_{R^n}T}$, and bound
the global approximation error in terms of the local error contributions,
is  sufficient (but not necessary) to yield a global error decaying as $\sqrt{n}\sigma_{n+1}$ or better. All of this is possible for the GFEM, which is why we employ this method in this subsection to build a global approximation from the local reduced spaces generated by Algorithm \ref{algo:adaptive_range_approximation}.
We refer to this numerical example as {\it Example 4}.
The convergence theory for the GFEM with randomized basis generation is given in the supplementary materials subsection SM5.2.

\label{sec:gfem description}
On $\Omega_{gl} = (0,1)^2$ we consider the following PDE: Find $u_{gl} \in X_{gl}=\{ v \in H^{1}(\Omega_{gl})\, : \, v = 0 \enspace \text{on} \enspace \Sigma_{D}\}$, such that
\begin{equation}
\label{eq:full problem}
-\divergence (k \nabla u_{gl}) = f \quad \text{in} \enspace X_{gl}',
\end{equation}
where $k \in L^{\infty}(\Omega_{gl})$, $0< k_{0} \leq k \leq k_{1} < \infty$
and $f(\varphi) := \int_{\Omega_{gl}} \widehat f \varphi \ dx$
for a source term $\widehat f \in L^2(\Omega_{gl})$.
We use the GFEM to compute an approximation of $u_{gl}$. To this end let
$\{\omega_{i}\}_{i=1}^{m}$
be an open cover of $\Omega_{gl}$ such that $\Omega_{gl}=\cup_{i=1}^{m}\omega_{i}$. 
For each domain, we define a local space
$
X_i := \big\{v|_{\omega_i} \ \big| \ v \in X_{gl} \big\}.
$
We will construct local reduced spaces $R^{n}_i \subset X_i$ and the global GFEM space
$
X_\mathrm{GFEM} := \bigoplus_{i=1,\dots,m} \big\{ \varrho_i v_i \ \big| \ v_i \in R^{n}_i \big\},
$
where $\varrho_i$ is a suitably defined partition of unity (see the supplementary materials subsection SM5.2).
The GFEM solution $u_\mathrm{GFEM} \in X_\mathrm{GFEM}$ is then defined as the solution of
$
-\divergence (k \nabla u_\mathrm{GFEM}) = f \quad \text{in} \enspace X_\mathrm{GFEM}'.
$
\label{sec:gfem basis generation}
\begin{figure}
\begin{minipage}[b]{0.25\textwidth}
\centering
\includegraphics[width=0.9\textwidth,frame]{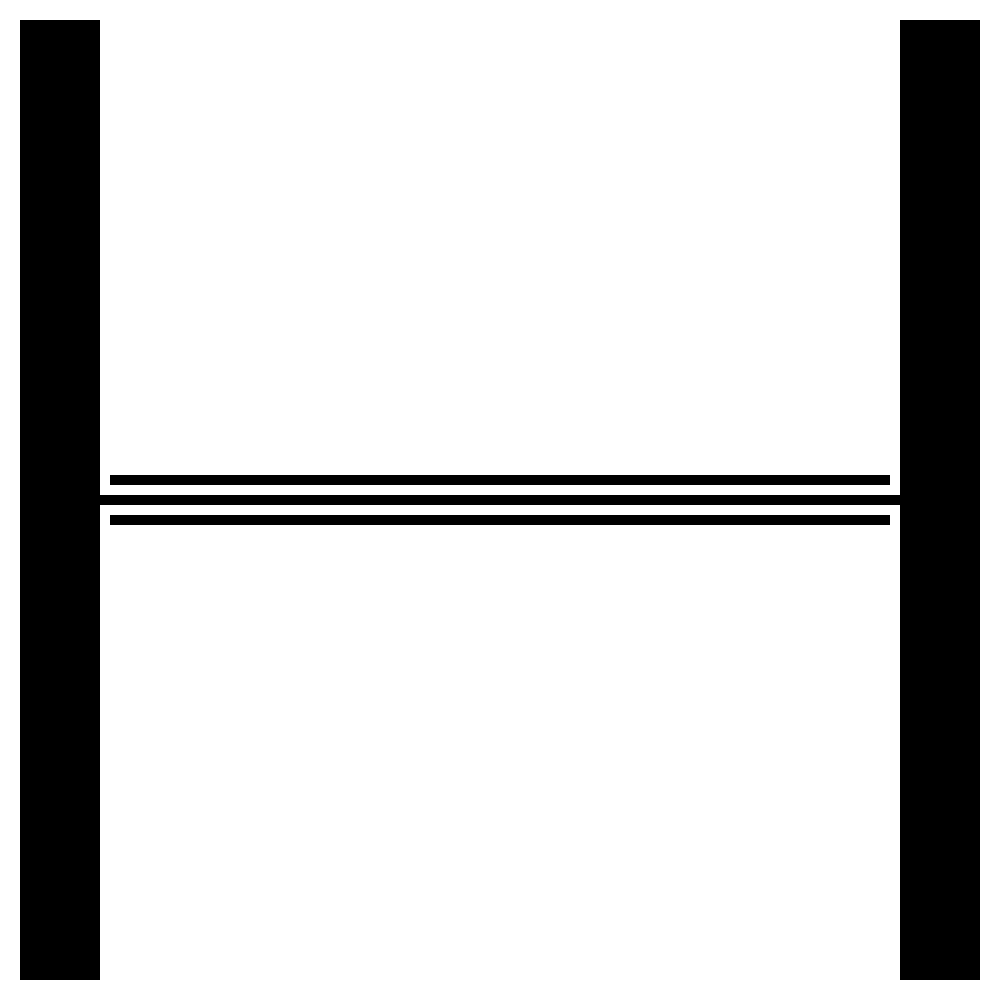}
\caption{\footnotesize Coefficient field for GFEM Example 2; white equates to 1 and black to $10^5$.
}
\label{fig:example 2}
\end{minipage}
\quad
\begin{minipage}[b]{0.7\textwidth}
\centering
\begin{tikzpicture}
\begin{semilogyaxis}[
    title={\it Example 4.1},
    width=0.52\textwidth,
    height=4cm,
    xmin=0,
    xmax=80,
    ymin=1e-13,
    ymax=1e10,
    xlabel=local basis size $n$,
    grid=both,
    grid style={line width=.1pt, draw=gray!20},
    major grid style={line width=.2pt,draw=gray!50},
    minor ytick={1e-15, 1e-14, 1e-13, 1e-12, 1e-11, 1e-10, 1e-9, 1e-8, 1e-7, 1e-6, 1e-5, 1e-4, 1e-3, 1e-2, 1e-1, 1e0, 1e1, 1e2, 1e3, 1e4, 1e5, 1e6, 1e7, 1e8, 1e9, 1e10},
    ytick={1e9,1e6,1e3, 1e0, 1e-3, 1e-6, 1e-9, 1e-12,1e-15},
    max space between ticks=25,
  ]
  \addplot+[thick, mark=triangle, mark repeat=10] table[x expr=\coordindex,y index=0]{maxsvd_poisson.dat};
  \addplot+[solid,          black, thick, mark=o, mark repeat=10] table[x index=0, y index=3] {local_global_errors_poisson.dat};
  \addplot+[solid,          red, thick, mark=x, mark repeat=10] table[x index=0, y index=8] {local_global_errors_poisson.dat};
\end{semilogyaxis}
\end{tikzpicture}
\begin{tikzpicture}
\begin{semilogyaxis}[
    title={\it Example 4.2},
    width=0.52\textwidth,
    height=4cm,
    xmin=0,
    xmax=80,
    ymin=1e-13,
    ymax=1e10,
    xlabel=local basis size $n$,
    yticklabels={,,},
    grid=both,
    grid style={line width=.1pt, draw=gray!20},
    major grid style={line width=.2pt,draw=gray!50},
    minor ytick={1e-15, 1e-14, 1e-13, 1e-12, 1e-11, 1e-10, 1e-9, 1e-8, 1e-7, 1e-6, 1e-5, 1e-4, 1e-3, 1e-2, 1e-1, 1e0, 1e1, 1e2, 1e3, 1e4, 1e5, 1e6, 1e7, 1e8, 1e9, 1e10},
    ytick={1e9,1e6,1e3, 1e0, 1e-3, 1e-6, 1e-9, 1e-12,1e-15},
    max space between ticks=25,
    legend pos= north east,
    legend style={at={(1.05,1.05)},anchor=north east},
  ]
  \addplot+[thick, mark=triangle, mark repeat=10] table[x expr=\coordindex,y index=0]{maxsvd_H.dat};
  \addplot+[solid, black, thick, mark=o, mark repeat=10] table[x index=0, y index=3] {local_global_errors_H.dat};
  \addplot+[solid, red, thick, mark=x, mark repeat=10] table[x index=0, y index=8] {local_global_errors_H.dat};
  \legend{svd,local,global}
\end{semilogyaxis}
\end{tikzpicture}
\caption{\footnotesize
The slowest SVD decay (svd), the maximum relative local error (local), and the relative global error (global) for the two GFEM examples ({\it Example 4}).
Median values over 100 realizations.
}
\label{fig:qualitative gfem}
\end{minipage}
\vspace{-15pt}
\end{figure}

We construct reduced spaces $R^{n}_i$, each approximating the solution $u_{gl}$ on $\omega_i$.
To this end, we introduce $\omega_{i}^{*}$, satisfying $\omega_{i}\subsetneq \omega_{i}^{*} \subset \Omega_{gl}$ with $\dist(\partial \omega_{i}^{*}\setminus \partial \Omega_{gl},\partial \omega_{i}) \geq \rho >0$, which denotes the oversampling domain used to construct the reduced space and thus corresponds to $\Omega$ in the remainder of this article. We denote its inner boundary $\partial \omega_i^* \setminus \partial \Omega_{gl}$ by $\Gamma_{out,i}$.
Denoting the space of $A$-harmonic functions on $\omega_i^*$ as $\tilde{H}_i$,
the transfer operator is defined as
$
T(w|_{\Gamma_{out,i}}) 
:= 
(w - P_{\ker(\mathcal{A}),\omega_i}(w))|_{\omega_i}$ 
for all $w \in \tilde{H}_i.
$
The spaces $S_i := \{ w|_{\Gamma_{out,i}} \ : \ w \in \tilde{H}_i\}$
are equipped with the $L^2$-inner product. 
In the range spaces
$R_i := 
\{ (w - P_{\ker(\mathcal{A}), \omega_i})|_{\omega_i} \, : \, w \in \tilde{H}_i\}
$
we use the energy inner product.
We apply Algorithm \ref{algo:adaptive_range_approximation} to these transfer operators using $n_t = 20$ test vectors throughout this subsection and a global maximum failure probability
of $\varepsilon_\mathrm{fail} = 10^{-15}$.
For the full GFEM algorithm with all details, see the supplementary materials subsection SM5.1.

\label{sec:gfem examples}
We demonstrate the GFEM with randomized basis generation on two examples.
The first example ({\it Example 4.1} in the following) is the
Poisson problem $- \Delta u = 1$, i.e.~$k\equiv 1$ and $\widehat f\equiv 1$. We consider this problem because the singular values and singular vectors of all transfer operators associated with subdomains $\omega_{i}^{*}$ that do not lie on the boundary of $\Omega_{gl}$ are the same and thus only the boundary has a (slight) influence. Therefore, we would expect that for this test case the convergence rate of the global error is similar to the convergence rate of $\|T-P_{R^{n}}T\|$. The second example ({\it Example 4.2} in the following) is more
complex and features small details, high contrast, and
high conductivity channels. In particular the solution of this example is non-smooth. For {\it Example 4.2}, we define a high
conductivity region
$\Omega_{hc} :=
\left[(0.02,0.1) \times (0.02, 0.98) \right] \cup
\left[(0.9, 0.98) \times (0.02, 0.98) \right] \cup
\left[(0.11, 0.89) \times (0.475, 0.485) \right] \cup
\left[(0.1, 0.9) \times(0.495, 0.505) \right] \cup
\left[(0.11, 0.89) \times (0.515, 0.525) \right]
$
and define
$k(x) := 10^5$ for $x \in \Omega_{hc}$ and $k(x) := 1$ else.
For the right hand side, we define a heating region $\Omega_{heat} := (0.9, 0.98) \times (0.02, 0.98)$
and a cooling region
$\Omega_{cool} := (0.02, 0.1) \times (0.02, 0.98)$
and define the right hand side as
$\widehat f(x) := 1$ for $x \in \Omega_{heat}$,
$\widehat f(x) :=-1$ for $x \in \Omega_{cool}$, and
$\widehat f(x) :=0 $ else.
For both examples, the domain is discretized using a regular mesh where
the domain is partitioned into $200\times 200$ squares of size $0.005^2$, 
each of which is divided into four triangles.
On this mesh, standard $P1$ basis functions are used,
spanning the FEM space. It has 80401 degrees of freedom,
of which 800 are constrained due to the Dirichlet boundary.
As local domains $\omega_i$ we use patches of size $0.2 \times 0.2$
with an overlap of size $0.1$. This accounts for
$9 \times 9 = 81$ subdomains $\omega_i$.
For the oversampling size we also use $0.1$, so the domains
$\omega_i^*$ in the interior have size $0.4\times 0.4$
while the domains $\omega_i^*$ at one boundary have size
$0.4\times 0.3$ or $0.3 \times 0.4$ and the domains
$\omega_i^*$ in the corners have size $0.3\times 0.3$.
The dimension of the source spaces $N_{S_i}$ and $N_{R_i}$
differ for domains in the interior, at the boundary
and in the corners. For domains in the interior,
it holds $N_{S_i} = 320$ and $N_{R_i} = 1681$.
\label{sec:gfem numerical results}
\begin{figure}
\centering
\begin{tikzpicture}
\begin{loglogaxis}[
    title={\it Example 4.1},
    x dir=reverse,
    width=0.38\textwidth,
    height=4.5cm,
    xmin=1e-11,
    xmax=1e5,
    ymin=1e-15,
    ymax=1e1,
    xlabel=$\varepsilon$,
    ylabel=$\min_{v_i \in R^n_i}\frac{\| k^{1/2} \nabla (u_{gl} - v_{i})\|_{L^{2}(\omega_{i})}}{\| k^{1/2}\nabla u_{gl} \|_{L^{2}(\omega_{i}^{*})}}$,
    grid=both,
    grid style={line width=.1pt, draw=gray!20},
    major grid style={line width=.2pt,draw=gray!50},
    xtick={1e-10, 1e-5, 1e0, 1e5},
    minor xtick={1e-10, 1e-9, 1e-8, 1e-7, 1e-6, 1e-5, 1e-4, 1e-3, 1e-2, 1e-1, 1e0, 1e1, 1e2, 1e3, 1e4},
    minor ytick={1e-15, 1e-14, 1e-13, 1e-12, 1e-11, 1e-10, 1e-9, 1e-8, 1e-7, 1e-6, 1e-5, 1e-4, 1e-3, 1e-2, 1e-1, 1e0, 1e1, 1e2, 1e3},
    max space between ticks=25,
  ]
  \addplot+[densely dotted, blue, thick,  mark=none] table[x index=0, y index=5] {local_approximations_poisson.dat};
  \addplot+[densely dashed, blue, thick,  mark=none] table[x index=0, y index=4] {local_approximations_poisson.dat};
  \addplot+[solid,          black, thick, mark=none] table[x index=0, y index=3] {local_approximations_poisson.dat};
  \addplot+[densely dashed, brown, thick, mark=none] table[x index=0, y index=2] {local_approximations_poisson.dat};
  \addplot+[densely dotted, brown, thick, mark=none] table[x index=0, y index=1] {local_approximations_poisson.dat};
\end{loglogaxis}
\end{tikzpicture}
\begin{tikzpicture}
\begin{loglogaxis}[
    title={\it Example 4.2},
    x dir=reverse,
    width=0.38\textwidth,
    height=4.5cm,
    xmin=1e-11,
    xmax=1e5,
    ymin=1e-15,
    ymax=1e1,
    legend pos=outer north east,
    xlabel=$\varepsilon$,
    yticklabels={,,},
    grid=both,
    grid style={line width=.1pt, draw=gray!20},
    major grid style={line width=.2pt,draw=gray!50},
    xtick={1e-10, 1e-5, 1e0, 1e5},
    minor xtick={1e-10, 1e-9, 1e-8, 1e-7, 1e-6, 1e-5, 1e-4, 1e-3, 1e-2, 1e-1, 1e0, 1e1, 1e2, 1e3, 1e4},
    minor ytick={1e-15, 1e-14, 1e-13, 1e-12, 1e-11, 1e-10, 1e-9, 1e-8, 1e-7, 1e-6, 1e-5, 1e-4, 1e-3, 1e-2, 1e-1, 1e0, 1e1, 1e2, 1e3},
    max space between ticks=25,
  ]
  \addplot+[densely dotted, blue, thick,  mark=none] table[x index=0, y index=5] {local_approximations_h.dat};
  \addplot+[densely dashed, blue, thick,  mark=none] table[x index=0, y index=4] {local_approximations_h.dat};
  \addplot+[solid,          black, thick, mark=none] table[x index=0, y index=3] {local_approximations_h.dat};
  \addplot+[densely dashed, brown, thick, mark=none] table[x index=0, y index=2] {local_approximations_h.dat};
  \addplot+[densely dotted, brown, thick, mark=none] table[x index=0, y index=1] {local_approximations_h.dat};
  \legend{max, 75 percentile, 50 percentile, 25 percentile, min};
\end{loglogaxis}
\end{tikzpicture}
\caption{\footnotesize Relative local error
$\min_{v_i \in R^n_i}\| k^{1/2} \nabla (u_{gl} - v_{i})\|_{L^{2}(\omega_{i})} / \| k^{1/2}\nabla u_{gl} \|_{L^{2}(\omega_{i}^{*})}$
versus target local error $\varepsilon$ for {\it Example 4}.
Reduced spaces $R^n_i$ generated with adaptive algorithm \ref{algo:adaptive_range_approximation}.
Statistics over 1000 samples and over all 81 local spaces.
}
\label{fig:gfem local approximation}
\vspace{-15pt}
\end{figure}
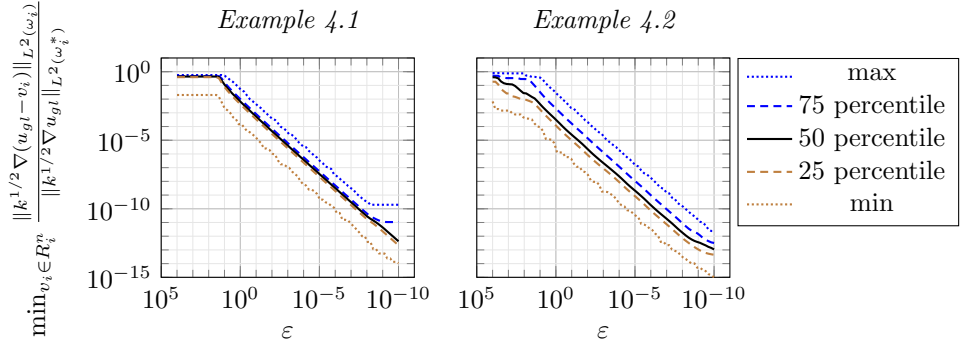

All 81 transfer operators in the two examples have an exponentially decaying spectrum.
The slowest decaying spectrum is shown in \cref{fig:qualitative gfem}
along with the maximum local relative approximation error and the 
global relative approximation error.
The maximum relative local approximation error follows the
spectrum of the transfer operator closely for {\it Example 4.1}. For {\it Example 4.2}, 
the behavior is similar.
As anticipated, this error decay propagates to the relative global approximation error, which
flattens out due to numerical effects at about $10^{-13}$ for {\it Example 4.1}
and $10^{-10}$ for {\it Example 4.2}.
To compute the spectrum of the transfer operators, the
numerically more accurate eigenvalue problem presented
in the supplementary materials section SM6 is used.
To further examine the behavior of the local error
$\min_{v_i \in R^{n}_i}\| k^{1/2} \nabla (u_{gl} - v_{i})\|_{L^{2}(\omega_{i})} / \| k^{1/2}\nabla u_{gl} \|_{L^{2}(\omega_{i}^{*})}$
when using the adaptive algorithm, we construct local approximation spaces $R^{n}_i$ with the proposed adaptive range recovery and measure the relative local error. Statistics over 1,000 different realizations and all 81 local spaces show
that results are more accurate than required by about 2.5 orders of magnitude for {\it Example 4.1} and about 3.5 orders of magnitude for {\it Example 4.2},
see \cref{fig:gfem local approximation}.
This discrepancy is in part caused by the fact that the adaptive range approximation generates spaces which are better than required,
as was already discussed in \cref{subsect:analytic,subsect:linearelasticity}. This accounts for about 1 to 1.5 orders of
magnitude. The other part is the pessimistic estimate for the local error given in the supplementary materials, Lemma SM5.2.
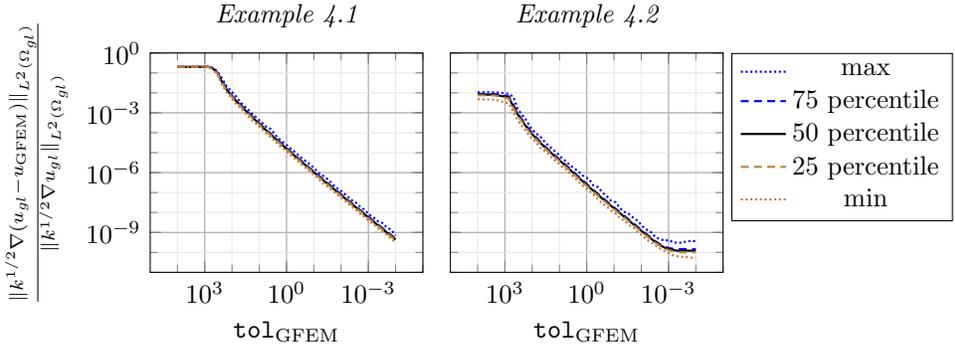
\begin{figure}
\centering
\begin{tikzpicture}
\begin{loglogaxis}[
    title={\it Example 4.1},
    x dir=reverse,
    width=0.4\textwidth,
    height=4.5cm,
    xmin=1e-5,
    xmax=1e5,
    ymin=1e-11,
    ymax=1e0,
    legend pos=outer north east,
    xlabel=$\algotol_\mathrm{GFEM}$,
    grid=both,
    grid style={line width=.1pt, draw=gray!20},
    major grid style={line width=.2pt,draw=gray!70},
    xtick={1e-12,1e-9,1e-6, 1e-3, 1e0, 1e3,1e6},
    ytick={1e-12,1e-9,1e-6, 1e-3, 1e0, 1e3,1e6},
    minor xtick={1e-10, 1e-9, 1e-8, 1e-7, 1e-6, 1e-5, 1e-4, 1e-3, 1e-2, 1e-1, 1e0, 1e1, 1e2, 1e3, 1e4},
    minor ytick={1e-15, 1e-14, 1e-13, 1e-12, 1e-11, 1e-10, 1e-9, 1e-8, 1e-7, 1e-6, 1e-5, 1e-4, 1e-3, 1e-2, 1e-1, 1e0, 1e1, 1e2, 1e3},
    max space between ticks=25,
    ylabel=$\frac{\| k^{1/2} \nabla (u_{gl} - u_\mathrm{GFEM})\|_{L^{2}(\Omega_{gl})}}{\|k^{1/2} \nabla u_{gl} \|_{L^{2}(\Omega_{gl})}}$,
  ]
  \addplot+[densely dotted, blue, thick,  mark=none] table[x index=0, y index=5] {global_approximations_poisson.dat};
  \addplot+[densely dashed, blue, thick,  mark=none] table[x index=0, y index=4] {global_approximations_poisson.dat};
  \addplot+[solid,          black, thick, mark=none] table[x index=0, y index=3] {global_approximations_poisson.dat};
  \addplot+[densely dashed, brown, thick, mark=none] table[x index=0, y index=2] {global_approximations_poisson.dat};
  \addplot+[densely dotted, brown, thick, mark=none] table[x index=0, y index=1] {global_approximations_poisson.dat};
\end{loglogaxis}
\end{tikzpicture}
\begin{tikzpicture}
\begin{loglogaxis}[
    title={\it Example 4.2},
    x dir=reverse,
    width=0.4\textwidth,
    height=4.5cm,
    xmin=1e-5,
    xmax=1e5,
    ymin=1e-11,
    ymax=1e0,
    legend pos=outer north east,
    xlabel=$\algotol_\mathrm{GFEM}$,
    yticklabels={,,},
    grid=both,
    grid style={line width=.1pt, draw=gray!20},
    major grid style={line width=.2pt,draw=gray!70},
    xtick={1e-12,1e-9,1e-6, 1e-3, 1e0, 1e3,1e6},
    ytick={1e-12,1e-9,1e-6, 1e-3, 1e0, 1e3,1e6},
    minor xtick={1e-10, 1e-9, 1e-8, 1e-7, 1e-6, 1e-5, 1e-4, 1e-3, 1e-2, 1e-1, 1e0, 1e1, 1e2, 1e3, 1e4},
    minor ytick={1e-15, 1e-14, 1e-13, 1e-12, 1e-11, 1e-10, 1e-9, 1e-8, 1e-7, 1e-6, 1e-5, 1e-4, 1e-3, 1e-2, 1e-1, 1e0, 1e1, 1e2, 1e3},
    max space between ticks=25,
  ]
  \addplot+[densely dotted, blue, thick,  mark=none] table[x index=0, y index=5] {global_approximations_h.dat};
  \addplot+[densely dashed, blue, thick,  mark=none] table[x index=0, y index=4] {global_approximations_h.dat};
  \addplot+[solid,          black, thick, mark=none] table[x index=0, y index=3] {global_approximations_h.dat};
  \addplot+[densely dashed, brown, thick, mark=none] table[x index=0, y index=2] {global_approximations_h.dat};
  \addplot+[densely dotted, brown, thick, mark=none] table[x index=0, y index=1] {global_approximations_h.dat};
  \legend{max, 75 percentile, 50 percentile, 25 percentile, min};
\end{loglogaxis}
\end{tikzpicture}
\caption{\footnotesize Relative global error
$\| k^{1/2} \nabla (u_{gl} - u_\mathrm{GFEM})\|_{L^{2}(\Omega_{gl})} / \|k^{1/2} \nabla u_{gl} \|_{L^{2}(\Omega_{gl})}$
versus target global error $\algotol_\mathrm{GFEM}$ for {\it Example 4}.
GFEM solution $u_\mathrm{GFEM}$ generated with Algorithm 1 in the supplementary materials,
reduced spaces $R^{n}_i$ generated with adaptive Algorithm \ref{algo:adaptive_range_approximation}.
Statistics over 1,000 samples.
}
\label{fig:global_approximation}
\vspace{-15pt}
\end{figure}
The error decay propagates to the global relative error
$\big(\| k^{1/2} \nabla (u_{gl} - u_\mathrm{GFEM})\|_{L^{2}(\Omega_{gl})}\big) \big/ \big(\|k^{1/2} \nabla u_{gl} \|_{L^{2}(\Omega_{gl})}\big).$
It is possible to choose a target error
$\algotol_\mathrm{GFEM}$ 
and choose all tolerances accordingly, so the resulting approximation will have at most
this relative error:
From $\algotol_\mathrm{GFEM}$ as a target maximum for the global relative error,
we calculate the maximum local relative error using Proposition SM5.1.
Using this maximum local relative error, we calculate a limit for the operator norm
$\norm{T_i - P_{R_i^n}T_i}$ using Lemma SM5.2.
This limit for the operator norm is then used to steer the adaptive range finder algorithm,
Algorithm \ref{algo:adaptive_range_approximation}.
 There are no unknown constants. For details,
see the supplementary materials subsection SM5.2.
The global relative error, shown in \cref{fig:global_approximation}, confirms this and is more accurate than required by about 4.5 orders of magnitude for {\it Example 4.1} and about 6.5 orders of magnitude for {\it Example 4.2}. From the local error to the global error, we loose about 2
orders of magnitude for {\it Example 4.1} and about 3 orders of magnitude for {\it Example 4.2}. This is due to the pessimistic estimate
in Proposition SM5.1.

\section{Conclusions}\label{sect:conclusions}

Recently, optimal local reduced spaces for localized MOR procedures have been proposed in \cite{BabLip11,SmePat16}. However, a straightforward FE approximation of those optimal local spaces is very expensive. In this article we have proposed an adaptive randomized range finder algorithm based on methods from randomized LA \cite{halko2011finding} that adaptively builds local reduced spaces for localized MOR procedures from local solutions of the PDE with random boundary conditions. Starting from results in randomized LA \cite{halko2011finding,Gor85,Gor88,CheDon05} we have shown that the randomly generated local reduced spaces produce an approximation with a convergence rate that is only slightly worse than the optimal rate; the rate is deteriorated by about the square root of the basis size. Finally, the adaptive randomized range finder algorithm is steered by a probabilistic a posteriori error estimator for which we have demonstrated its efficiency. 

The numerical experiments show that the local spaces constructed by the adaptive randomized range finder algorithm indeed converge with a nearly optimal rate. It can also be seen that the a priori error bound seems to be sharp in the sense that for some numerical experiments the projection error converges exactly as predicted by the a priori error bound.
Moreover, we observed that after a preasymptotic regime the convergence behavior of the projection error is independent of the mesh size, indicating that it might be possible to generalize the results of the present paper to the continuous setting. This is the subject of future work. For the GFEM we exemplarily demonstrated in the numerical experiments that the excellent local approximation behavior of the randomly generated spaces carries over to the global level for an example with high conductivity channels. Regarding the probabilistic a posteriori error estimator the numerical experiments have showed also for a transfer operator with slowly decaying singular values and discrete source and range spaces of rather large dimensions that we only need $10$ to $20$ test vectors in order to obtain a sharp bound. Thanks to both the nearly optimal approximation capacities of the randomly generated local reduced spaces and the good effectivity of the probabilistic a posteriori error estimator also for few numbers of test vectors the adaptive randomized range finder algorithm requires only very few local solutions of the PDE in addition and therefore has a close to optimal computational complexity
and is faster than the calculation of the optimal spaces using ARPACK on the corresponding eigenvalue problem.

The extension of the proposed method to transient and nonlinear problems is the subject of future work. We note that although prescribing random boundary conditions and using the solution of the nonlinear PDE evaluated either on an interface or a subdomain might actually yield a reduced space with decent approximation properties, we believe that the corresponding numerical analysis might however be quite involved.

\subsection*{Acknowledgments} We would like to thank Dr. Jonas Ballani of Akselos for the fruitful discussion on randomized linear algebra at the MoRePaS workshop 2015 in Trieste. Moreover, we are grateful to Dr.~Clemens Pechstein from CST AG for pointing us to the numerically more accurate eigenvalue problem shown the supplementary materials section SM6.

\bibliographystyle{siamplain}
\bibliography{paper.bib}
\end{document}